\documentclass[11pt]{amsart}

\usepackage[latin9]{luainputenc}
\usepackage{amsfonts}
\usepackage{amsmath}
\usepackage{amstext}
\usepackage{amsthm}
\usepackage{amssymb}
\usepackage{braket}
\usepackage{bm}
\usepackage{cancel}
\usepackage{caption}
\usepackage{color}
\usepackage{enumerate}
\usepackage{enumitem}
\usepackage{esint}
\usepackage{extarrows}
\usepackage[a4paper,margin=1in,footskip=0.25in]{geometry}
\usepackage{graphicx}
\usepackage{longtable}
\usepackage{mathrsfs}
\usepackage{mathtools}
\usepackage{mdframed}
\usepackage{soul}
\usepackage{subcaption}
\usepackage{tikz}
\usepackage{tikz-cd}
\usepackage{xcolor}
\usepackage{relsize}
\usepackage{dsfont}

\numberwithin{equation}{section}
\graphicspath{ {./images/} }
\makeatletter
\@namedef{subjclassname@2020}{\textup{2020} Mathematics Subject Classification}
\makeatother

\newtheorem{defi}{Definition}[section]
\newtheorem{lem}[defi]{Lemma}
\newtheorem{theo}[defi]{Theorem}
\newtheorem{cor}[defi]{Corollary}
\newtheorem{pro}[defi]{Proposition}

\newtheorem{rem}[defi]{Remark}

\DeclareMathOperator{\divop}{div}

\DeclareMathOperator{\diam}{diam}

\DeclareMathOperator{\RR}{\mathbb{R}}

\newcommand{\cT}{{\mathcal T}}

\allowdisplaybreaks 

\title[Cucker-Smale model with singular matrix-valued communication]{Measure solutions to a kinetic Cucker-Smale model with singular and matrix-valued communication}

\author{Jan Peszek}
\address{Institute of Applied Mathematics and Mechanics, University of Warsaw, ul. Banacha 2, 02-097 Warszawa, Poland}
\email{j.peszek@mimuw.edu.pl}

\author{David Poyato}
\address{Institut Camille Jordan (ICJ), UMR 5208 CNRS \& Universit\'e Claude Bernard Lyon 1, 69100 Villeurbanne, France, and Research Unit ``Modeling Nature'' (MNat), Universidad de Granada, Granada, 18071, Spain}
\email{poyato@math.univ-lyon1.fr}

\begin{document}

\date{\today}


\thanks{\textbf{Acknowledgment.}  This work has been partially supported by the Polish National Science Centre grant No 2018/31/D/ST1/02313 (SONATA) (JP), by the European Research Council under the European Union's Horizon 2020 research and innovation program grant agreement No 865711 (DP), and by the Spanish projects RTI2018-098850-B-I00 from the MINECO-Feder and P18-RT-2422 \& A-FQM-311-UGR18 from the Junta de Andalucia (DP)}

\begin{abstract}
 We introduce a multi-dimensional variant of the kinetic Cucker-Smale model with singular and matrix-valued communication weight, which reduces to the singular kinetic Cucker-Smale equation in the one-dimensional case. We propose an appropriate notion of weak measure-valued solution to this second-order system and a suitable first-order reduction, which persist beyond the blow-up time in classical norms. The core of the paper is to show that both formulations are equivalent in this singular regime, thus extending the previous results from regular to weakly singular communication weights. As a consequence, we obtain: (i) global-in-time well-posedness of weak measure-valued solutions, (ii) quantitative convergence rates to equilibrium, and  (iii) uniform-in-time mean-field limits thanks to the above-mentioned equivalence and our recent work on fibered gradient flows.
\end{abstract}

\maketitle

\section{Introduction}

Mathematical models of collective dynamics are widely used to describe a variety of phenomena ranging from aggregation of bacteria and flocking of birds to opinion dynamics, distribution of goods and traffic flows. They serve as an inspiration for new directions of mathematical development. Among such models is the Cucker-Smale (CS) model of self-propelled particles with a tendency to flock \cite{CS-07}. It has garnered significant attention in the last decade owing to its relative simplicity and an interesting structure reminiscent of classical models of fluid mechanics. It has been studied in a variety of directions including asymptotics and large-time behavior \cite{HL-09, PKH-10, PGE-09}, influence of additional forces \cite{BH-17, BCC-11, HHK-10} and control \cite{AKMO-20, BBCK-18, CKPP-19}. Other directions include collision-avoidance \cite{CCMP-17, CD-10, D-21} and propagation of finite-range interactions \cite{J-18, MPT-19}.

Throughout this paper, we consider the following agent-based CS-type system of alignment dynamics with matrix-valued communication
\begin{align}\label{sysp}
\begin{aligned}
&\dot{x}_i=v_i,\quad  i\in\{1,..., N\},\\
&\dot{v}_i\displaystyle=\frac{1}{N}\sum_{j=1}^N D^2 W(x_i-x_j)(v_j-v_i),\\
&(x_i(0),v_i(0))  = (x_{i0},v_{i0}) \in \RR^{2d},
\end{aligned}
\end{align}
where $N$ is the total number of agents, while $x_i(t),v_i(t)\in \mathbb{R}^d$ denote the position and velocity of the $i$th agent at time $t\geq 0$, respectively. Above, $W$ is a super-linear kernel given by
\begin{equation*}
    W(x)=\frac{1}{2-\alpha}\frac{1}{1-\alpha}|x|^{2-\alpha}, \quad x\in \mathbb{R}^d,
\end{equation*}
for the parameter $\alpha\in(0,1)$. Therefore, its first and second-order derivatives take the form
\begin{align*}
\nabla W(x) &= \frac{1}{1-\alpha}\phi(|x|)\,x,\\
D^2 W(x) &= \frac{1}{1-\alpha}\phi(|x|)\left(I_d-\alpha \frac{x}{|x|}\otimes \frac{x}{|x|}\right),
\end{align*}
where $\phi$ can be viewed as a weakly singular communication weight defined as
\begin{equation}\label{ker}
\phi(r):=\frac{1}{r^\alpha},\quad r>0,\ \alpha\in(0,1).
\end{equation}

 Alignment models with matrix-valued communication similar to \eqref{sysp} have been recently studied by {\sc Shu} and  {\sc Tadmor} \cite{ST-21} and by {\sc Kim} \cite{K-22} and can be viewed as multi-D generalization of the 1D CS flocking model and an alternative to the classical multidimensional CS model \cite{CS-07}.
The large-crowd regime with $N\gg 1$ in \eqref{sysp} is often expressed by a kinetic model. For the classical CS model with regular communication weight (including some singular cases), the kinetic model can be derived as mean-field limit of \eqref{sysp} as $N\rightarrow \infty$, see for instance \cite{HL-09, MP-18}. Our approach is to perform the mean-field limit on the so-called first-order reduction of \eqref{sysp}, leading ultimately to the following kinetic equation
\begin{align}\label{sysk}
\begin{aligned}
&\partial_tf + v\cdot \nabla_xf + {\rm div}_v(\boldsymbol{F}[f]\,f) = 0,\quad t\geq 0,\ (x,v)\in \mathbb{R}^{2d},\\
&\boldsymbol{F}[f_t](x,v) :=  \int_{\RR^{2d}} D^2W(x-x')\,(v'-v)\, f_t(x',v')\,dx'\,dv',\\
&f_{t=0} = f_0,
\end{aligned}
\end{align}
 where $\boldsymbol{f}=f_t(x,v)$ is the density of particles with position $x\in \RR^d$ and velocity $v\in\RR^d$ at time $t\geq 0$. With $d=1$ system \eqref{sysp} is trivially reduced to the 1D weakly singular CS model
\begin{align}\label{csp}
\begin{aligned}
&\dot{x}_i=v_i,\quad i\in \{1,\ldots,N\},\\
&\dot{v}_i=\frac{1}{N}\sum_{j=1}^N\phi(|x_j-x_i|)(v_j-v_i),\\
& (x_i(0),v_i(0))=(x_{i0},v_{i0}),
\end{aligned}
\end{align}
whose associated kinetic model is \eqref{sysk} with $\boldsymbol{F}$ replaced by 
\begin{equation*}
    \boldsymbol{F}_{\text{CS}}[f_t](x,v):= \int_{\mathbb{R}^{2d}} \phi(|x-x'|)(v'-v)\,f_t(x',v')\,dx'\,dv'.
\end{equation*}
Thus, the systems \eqref{sysp} and \eqref{sysk} can be viewed as a multidimensional generalization of the 1D weakly singular CS model. 

The goal of the paper is twofold. First we propose a notion of weak measure-valued solutions to \eqref{sysk} that persist after eventual emergence of singularities, and we prove that under suitable assumptions they are equivalent to weak measure-valued solutions of a multidimensional Kuramoto-type equation investigated recently in \cite{PP-22-1-arxiv}. Then, employing the machinery developed in \cite{PP-22-1-arxiv} and earlier, by the second author in \cite{P-19-arxiv}, we discuss global-in-time well-posedness and uniform stability estimates for system \eqref{sysk}, in particular obtaining convergence to the equilibrium and the mean-field limit.
This contribution falls in line with the mainstream drive to prove well-posedness for the singular CS model in its mesoscopic, kinetic variant. What makes our work distinguishable (particularly from \cite{K-22} discussed below) is that we are the first to do it for weakly singular interactions \eqref{ker} despite the fact that collision-avoidance does not hold in this regime \cite{P-14}, and our approach is applicable to the multidimensional case and provides actual equivalence on the level of weak measure-valued formulations.

\subsection{Main result}

 The starting point of our considerations is  the change of variables
\begin{align}\label{trans}
\omega = v+\nabla W*\rho_t(x),\quad \mbox{with}\quad \rho_t(x):=\int_{\RR^d} f_t(x,v)\,dv,
\end{align}
which transforms \eqref{sysk} into the kinetic Kuramoto-type equation
\begin{align}\label{kurak}
\begin{aligned}
&\partial_t\mu + {\rm div}_x(\boldsymbol{u}[\rho]\,\mu) = 0,\quad t\geq 0,\ (x,\omega)\in \mathbb{R}^{2d},\\
&\boldsymbol{u}[\rho_t](x,\omega):=\omega-\nabla W*\rho_t(x),\quad \rho_t(x):=\int_{\mathbb{R}^d}\mu_t(x,\omega)\,d\omega,\\
&\mu_{t=0}=\mu_0,
\end{aligned}
\end{align}
satisfied by the probability density
\begin{align*}
\mu_t(x,\omega): = f_t(x,\omega-\nabla W*\rho_t(x)),\quad t\geq 0,\ (x,\omega)\in \mathbb{R}^{2d}.
\end{align*}
From a rigorous point of view, the above transformation and definition of $\boldsymbol{\mu}$ only makes sense under the assumption that everything is smooth, but we define it below in Definition \ref{D-change-variables} for general Radon measures as a push-forward of $\boldsymbol{f}$ along the mapping
$$\mathcal{T}^{2\to 1}[\rho_t](x,v):=(x,v+\nabla W*\rho_t(x)), \quad t\geq 0,\ (x,v)\in \mathbb{R}^{2d},$$
and, conversely, we define $\boldsymbol{f}$ as a push-forward of $\boldsymbol{\mu}$ along
$$\mathcal{T}^{1\to 2}[\rho_t](x,\omega):=(x,\omega-\nabla W*\rho_t(x)), \quad t\geq 0,\ (x,\omega)\in \mathbb{R}^{2d}.$$

Note that in the continuity equation in \eqref{kurak} the divergence is taken only with respect to the variable $x$. Such type of problems have been extensively studied by the authors recently in \cite{PP-22-1-arxiv}, wherein similar continuity equations were recognized as gradient flows with respect to the so-called {\it fibered Wasserstein metric}. From the perspective of our present considerations the key observation remarked in \cite{PP-22-1-arxiv} is that each time-slice $\mu_t$ of any solution $\boldsymbol{\mu}$ to \eqref{kurak} can be disintegrated with respect to its $\omega$-marginal $\nu$ and represented as
\begin{equation}\label{disint}
    \int_{\RR^{2d}}\psi(x,\omega)\,d\mu_t(x,\omega) = \int_{\RR^{d}} \left(\int_{\RR^d}\psi(x,\omega)\,d\mu_t^\omega(x)\right)d\nu(\omega),
\end{equation}
for all Borel-measurable functions $\psi:\RR^{2d}\longrightarrow [0,+\infty)$, see Theorem \ref{dis} in the appendix. In addition, one has that the $\omega$-marginal $\nu$ does not change in time, and the family of probability measures $\{\mu_t^\omega\}_{\omega\in\RR^d}$ are defined $\nu$-almost everywhere and verify
\begin{align*}
&\partial_t\mu^\omega+\divop_x(\boldsymbol{u}[\rho](\cdot,\omega)\,\mu^\omega),\quad t\geq 0,\ x\in \mathbb{R}^d,\\
&\boldsymbol{u}[\rho_t](x,\omega)=\omega-\int_{\mathbb{R}^d}\nabla W*\mu_t^{\omega'}(x)\,d\nu(\omega'),\\
&\mu_{t=0}^\omega=\mu_0^\omega,
\end{align*}
for $\nu$-a.e. $\omega\in \mathbb{R}^d$. Therefore, roughly speaking, \eqref{kurak}  amounts to a family of continuity equations depending on a parameter $\omega$, which is distributed according to the prescribed probability measure $\nu$. However, we remark that the interplay between different $\omega$-fibers of $\{\mu_t^\omega\}_{\omega\in\RR^d}$ is highly complex since the velocity field $\boldsymbol{u}[\rho_t]$ is reconstructed as an average of the combined effects of all fibers simultaneously. Keeping in mind the special role of the marginal $\nu$ we state our main result.

\begin{theo}[Equivalence]\label{main1}
 Let $T>0$, $\alpha\in(0,1)$ be given.  Let $\boldsymbol{\mu}$ be a weak measure-valued solution to $\eqref{kurak}$ in the sense of Definition \ref{D1st} (see below) with a prescribed $\omega$-marginal $\nu$. Then $f_t = (\cT^{1\to 2}[\rho_t])_{\#}\mu_t$  is a solution to \eqref{sysk} in the sense of Definition \ref{DCS} (see below) provided that
\begin{align}\label{nudiag0}
    \int_{\RR^{2d}}|\omega-\omega'|^\frac{2-3\alpha}{1-\alpha}\, d(\nu\otimes\nu)<\infty.
\end{align}
Conversely, if $\boldsymbol{f}$ is a solution to \eqref{sysk} in the sense of Definition \ref{DCS}, then $\mu_t = (\cT^{2\to 1}[\rho_t])_{\#}f_t$ has a time-independent probabilistic $\omega$-marginal $\nu$ and it is a solution to $\eqref{kurak}$ in the sense of Definition \ref{D1st} provided that the a priori estimate
\begin{align}\label{nudiagf0}
    \sup_{t\in[0,T]}\int_{\RR^{4d}}\left|(v+\nabla W*\rho_t(x))-(v'+\nabla W*\rho_t(x'))\right|^\frac{2-3\alpha}{1-\alpha}\, d(f_t\otimes f_t)< +\infty
\end{align}
holds true.
\end{theo}

Together with well-posedness results from \cite{PP-22-1-arxiv} such as Theorem 5.3 or Theorem B, the above equivalence yields the following theorem.

\begin{theo}[Well posedness]\label{main2}
Let $f_0\in {\mathcal P}_2(\RR^{2d})$ satisfy the assumption
\begin{equation}\label{nudiag0t0}
    \int_{\RR^{4d}}\left|(v+\nabla W*\rho_0(x))-(v'+\nabla W*\rho_0(x'))\right|^\frac{2-3\alpha}{1-\alpha}\, d(f_0\otimes f_0)< +\infty.
\end{equation}
Then, there exists a solution to \eqref{sysk} in the sense of Definition \ref{DCS}. This solution is unique in the class of solutions satisfying additionally \eqref{nudiagf0} in $[0,T]$.

\end{theo}

\begin{rem}

$\,$

(i) As we prove later, solutions $\boldsymbol{\mu}$ and $\boldsymbol{f}$ in the sense of Definitions  \ref{D1st} and \ref{DCS}, have uniformly bounded second order moments in $[0,T]$. Thus, for $\alpha\in(0,\frac{2}{3}]$ assumptions \eqref{nudiag0} and \eqref{nudiagf0} in Theorem \ref{main1} are trivially satisfied, since then the exponent $\frac{2-3\alpha}{1-\alpha}$ is positive and  \eqref{nudiag0} and \eqref{nudiagf0} can be obtained by interpolating between zeroth and second order moments. Indeed, by the definition of $W$ we have the following behavior
$$|v+\nabla W*\rho_t(x)|^{\frac{2-3\alpha}{1-\alpha}}\lesssim 1+|v|^{\frac{2-3\alpha}{1-\alpha}}+|x|^{2-3\alpha}\lesssim 1+|x|^2+|v|^2,$$
where in the last inequality we use Young's inequality together with the fact that $\frac{2-3\alpha}{1-\alpha}\in(0,2)$ for $\alpha\in (0,\frac{2}{3})$ (and thus also $2-3\alpha$ belongs to $(0,2)$).

\medskip

(ii) In light of the previous paragraph, condition \eqref{nudiagf0} in Theorem \ref{main1}  is relevant only when $\alpha\in (\frac{2}{3},1)$. Even then, it does not play any role in the existence of solutions, for which condition \eqref{nudiag0t0} in Theorem \ref{main2} is sufficient. The a priori condition is significant only as far as uniqueness is concerned. The possibility of existence of solutions $\boldsymbol{f}$ to the alignment system \eqref{sysk} that satisfy Definition \ref{DCS} but do not originate from solutions $\boldsymbol{\mu}$ in the sense of Definition \ref{D1st} is uncertain and requires further study. In such a case the novel conditions \eqref{nudiag0} and \eqref{nudiagf0} with negative exponent seem to play the key role.

\medskip

(iii) An alternative approach could have been to simply establish $f_t = (\cT^{1\to 2}[\rho_t])_{\#}\mu_t$ with $\boldsymbol{\mu}$ satisfying Definition \ref{D1st} as a notion of solutions to the alignment system \eqref{sysk}, as it was done for instance by {\sc Choi} and {\sc Zhang} in \cite{CZ-21}. This is further justified by the fact that in the class of so-called atomic solutions (which, in 1D, were shown to be unique in \cite{MP-18}) and in the class of smooth solutions it is easy to show equivalence between $f_t = (\cT^{1\to 2}[\rho_t])_{\#}\mu_t$ and $\boldsymbol{f}$ satisfying Definition \ref{DCS}. Our approach tackles a more difficult problem of equivalence of Definitions \ref{DCS} and \ref{D1st} for a wider class of solutions. It also leads to emergence of, otherwise unexpected nuances, such as assumption \eqref{nudiag0}. 

\end{rem}

The last result is a direct application of results from \cite{PP-22-1-arxiv} to solutions of \eqref{sysk}, which is possible thanks to equivalence granted by Theorem \ref{main1}.

\begin{theo}[Stability with compactly supported initial data]\label{main3}
Suppose that $\boldsymbol{f}$ is a solution to \eqref{sysk} satisfying the assumptions of Theorem \ref{main2} subject to initial datum $f_0$ satisfying conditions
\begin{align}\label{diamsupp}
\begin{aligned}
    &{\mathcal D}_x := \diam ({\rm supp}_x f_0)<+\infty,\\
    &{\mathcal D}_\omega := \diam \big(\{v+\nabla W*\rho_0(x):\ (x,v)\in {\rm supp} f_0\}\big)<+\infty.
\end{aligned}
\end{align}
Then, there exists a unique equilibrium $f_\infty\in \mathcal{P}_{2}(\mathbb{R}^{2d})$ of \eqref{sysk} such that
\begin{equation}\label{main30}
\int_{\mathbb{R}^{2d}}x\,df_0(x,v)=\int_{\mathbb{R}^{2d}}x\, df_\infty(x,v),
\end{equation}
and we also have
\begin{equation}\label{main31}
    W_{2}(f_t,f_\infty)\leq Ce^{-4(1-\alpha)\phi(D_1)t},
\end{equation}
for all $t\geq 0$. Here, $W_2$ denotes the $2$-Wasserstein metric, $C>0$ is a constant depending on the initial datum $f_0$ and on $f_\infty$, and $D_1>0$ is given by $$D_1:=\max\left\{\mathcal{D}_x,\mathcal{D}_\omega^{\frac{1}{1-\alpha}}\right\}.$$

Moreover let $\widetilde{\boldsymbol{f}}$ be another solution of \eqref{sysk},  subject to initial datum $\widetilde{f}_0$, satisfying the assumptions of Theorem \ref{main2} and condition \eqref{diamsupp} with constants $\widetilde{\mathcal{D}}_x$ and $\widetilde{\mathcal{D}}_\omega$. If in addition
$$\int_{\RR^{2d}}x\,df_0(x,v) = \int_{\RR^{2d}}x \,d\widetilde{f}_0(x,v), $$
then there exists an optimal transport-based metric ${\rm dist}$ on ${\mathcal P}_2(\RR^{2d})$ such that
$${\rm dist}(f_t,\widetilde{f}_t)\leq \left(1+\frac{1}{2\phi(D_2)}\right)\, {\rm dist}(f_0,\widetilde{f}_0),$$
for all $t\geq 0$, where $D_2$ is given by
$$D_2:=\max\left\{\mathcal{D}_x,\widetilde{\mathcal{D}}_x,\mathcal{D}_\omega^{\frac{1}{1-\alpha}},\widetilde{\mathcal{D}}_\omega^{\frac{1}{1-\alpha}}\right\}.$$
\end{theo}

Our setting lies in the intersection of two major directions of research related to the CS model: analysis of the CS model with singular interactions and study of its first-order reduction.
\subsubsection*{Cucker-Smale model with singular interactions}
Initial contributions towards the well posedness for the CS particle system with weakly singular interactions \eqref{ker} have been delivered by the first author in \cite{P-14, P-15}, which culminated in existence through the mean-field limit and the so called weak-atomic uniqueness for the kinetic CS model with singularity of order $\alpha\in(0,\frac{1}{2})$, see \cite{MP-18}. Other directions include the strongly singular CS model and its relation to collision-avoidance \cite{CCMP-17, M-18, YYC-20} and to the Euler-alignment system \cite{DMPW-19, DKRT-18,  ST-17, ST-18}. More information on the CS model can be found in the surveys \cite{CHL-17, MMPZ-19}. Outside of the aforementioned surveys is the study of monokinetic solutions to the kinetic CS model \cite{CC-21, KK-20}, which is an important new direction related to strongly singular CS model and the Euler-alignment system. Thus, from the perspective of kinetic CS model with singular interactions existence is known for $\alpha\in(0,\frac{1}{2})$ in multi-D and for $\alpha\in(0,1)$ in 1D, see \cite{CZ-21}. Here, we extend these results to $\alpha\in(0,1)$ in multi-D for the kinetic CS model with matrix-valued communication.

\subsubsection*{First-order reduction of the Cucker-Smale model} In line with the above-mentioned publication \cite{CZ-21}, we arrive at a new research direction related to transformation \eqref{trans}. In 1D it reduces to the well-known 1D Cucker-Smale-to-Kuramoto transformation, which has a very intuitive origin in the case of the particle system. Indeed, a single integration with respect to time $t\geq 0$ of system \eqref{csp} leads to the 1D Kuramoto-type equation
\begin{align}\label{kurap}
\begin{aligned}
&\dot{x}_i =  \omega_i + \frac{1}{N}\sum_{j=1}^N\nabla W(x_j-x_i),\\
&x_i(0)=x_{i0},\quad i\in\{1,\ldots,N\}.
\end{aligned}
\end{align}
Here, $\omega_i$ denote the fixed natural velocities defined as
$$
\omega_i := v_{i0} - \frac{1}{N}\sum_{j=1}^N \nabla W(x_{j0}-x_{i0}),
$$
which are nothing but a discrete variant of the change of variables \eqref{trans}. The above idea is well known and was used for instance in \cite{HKPZ-19, ZZ-20} to establish a complete control over cluster formation for the 1D CS model, and in \cite{P-14} to investigate the finite-time alignment and sticking of particles for the weakly singular CS model.
To the best of our knowledge, the current research employing the strategy of treating 1D CS model as a first-order Kuramoto-type system differs from our work in the following ways. First, in the singular 1D CS model the recent results such as the aforementioned \cite{CZ-21}  treat the weak formulation to \eqref{kurak} as relaxation of the weak formulation to \eqref{sysk}. In particular it is unclear if all solutions to \eqref{sysk} can be recovered as solutions to \eqref{kurak}. We provide conditions ensuring equivalence between their corresponding weak formulations leading to well-posedness. Second, we tackle the more general problem of the multidimensional system \eqref{sysk}. It is worthwhile to note that a similar approach was introduced recently by {\sc Kim} in \cite{K-21, K-22} but in the less complex case of regular communication weight.

One of the goals of this paper is to produce rigorous equivalence between first- and second-order systems for classes of weak solutions compatible with our recent work \cite{PP-22-1-arxiv}. In \cite{PP-22-1-arxiv} we introduced the concept of gradient flows with respect to the so-called fibered Wasserstein distance. We obtained a variety of results on stability for equation \eqref{kurak} including contractivity and its relation to asymptotics and the uniform-in-time mean-field limit for solutions initiated in compactly supported data. These results translate to equation \eqref{sysk}; some of these translated results can be seen in Theorem \ref{main3}.
Results similar to those of \cite{PP-22-1-arxiv} have been also obtained using the Filipov theory \cite{PPS-21, P-19-arxiv} for a Kuramoto-type model including a singular coupling, and they are also mostly compatible with the present paper.

The reminder of the paper is organized as follows. In Section \ref{sec:prelim} we present the necessary preliminary information including weak formulations of \eqref{sysk} and \eqref{kurak} and the rigorous definition of the transformation between them. Section \ref{sec:energest} is dedicated to the proof of energy estimate for weak solutions to \eqref{kurak}, while Section \ref{sec:mainproof} culminates in the proof of the main theorems.

\subsection*{Notation} Throughout the paper ${\mathcal P}(\RR^d)$ denotes the space of probability measures in $\RR^d$. Except otherwise stated, this space is equipped with the {\it narrow topology} (and we denote it by ${\mathcal P}(\RR^d)$-narrow), which can be defined sequentially as follows: a sequence $\{\mu_n\}$ in ${\mathcal P}(\RR^d)$ converges narrowly to $\mu\in{\mathcal P}(\RR^d)$ if and only if
\begin{align*}
    \lim_{n\to\infty}\int_{\RR^d} \psi\,d\mu_n = \int_{\RR^d} \psi\,d\mu, 
\end{align*}
for all bounded continuous functions $\psi$. 

If ${\mathcal T}:\RR^d\to\RR^d$ is a Borel-measurable mapping and $\mu\in{\mathcal P}(\RR^d)$, then by ${\mathcal T}_\#\mu\in{\mathcal P}(\RR^d)$ we denote the {\it pushforward} of $\mu$ along ${\mathcal T}$. It is defined on Borel sets $B$ as ${\mathcal T}_\#\mu(B) = \mu({\mathcal T}^{-1}(B))$ and obeys the change of variables formula
\begin{align*}
    \int_{\RR^d} g\,d {\mathcal T}_\#\mu = \int_{\RR^d} g \circ {\mathcal T} \,d\mu,
\end{align*}
for all Borel-measurable functions $g$. In particular, we have that $g$ is ${\mathcal T}_\#\mu$-integrable if, and only if, $g\circ {\mathcal T}$ is $\mu$-integrable.

Similarly, ${\mathcal P}_2(\RR^d)$ corresponds to the subspace of $\mathcal{P}(\mathbb{R}^d)$ having finite second order moment. For any two measures $\mu_1,\mu_2\in {\mathcal P}_2(\RR^d)$ we define  the Wasserstein $W_2$ distance as the solution of the following optimal transportation problem
\begin{align}\label{wasdef}
W_2(\mu_1, \mu_2) = \left(\inf_{\gamma\in \Gamma(\mu_1,\mu_2)}\int_{\RR^{2d}}\vert x-x'\vert^2\,d\gamma(x,x')\right)^{1/2},
\end{align}
where the infimum is taken over all probability measures $\gamma=\gamma(x,x')\in {\mathcal P}(\RR^{2d})$ with the $x$-marginal equal to $\mu_1$ and the $x'$-marginal equal to $\mu_2$. For more information on basics of optimal transport we refer to the textbooks \cite{AGS-08} or \cite{V-09}.

Throughout the paper, variables $(x,v)$ always relate to the second-order alignment system \eqref{sysk}, while $(x,\omega)$ are relegated to the first order Kuramoto-type equation \eqref{kurak}. Following this distinction, we define the diagonal set
\begin{equation}\label{diagonalset}
\begin{split}
        \Delta &:= \big\{((x,v),(x',v'))\in\RR^{4d}:\, (x,v)=(x',v')\big\}\\
        &\quad \mbox{ or } \big\{((x,\omega),(x',\omega'))\in\RR^{4d}:\, (x,\omega)=(x',\omega')\big\},
\end{split}
\end{equation}
and collision set
\begin{equation}\label{colisionset}
\begin{split}
          {\rm Col} &:= \{(t,(x,v),(x',v'))\in [0,T]\times\RR^{4d}:\ x=x'\mbox{ and } v\neq v'\}\\
           &\quad\mbox{ or } \{(t,(x,\omega),(x',\omega'))\in [0,T]\times\RR^{4d}:\ x=x'\mbox{ and } \omega\neq\omega'\},
\end{split}
\end{equation}
depending on whichever setting \eqref{sysk} or \eqref{kurak} is considered. While \eqref{diagonalset} and \eqref{colisionset} are a slight abuse of notation, it is contextually clear which is viable at any given moment. Furthermore, they are in fact equivalent since, by \eqref{trans}, $v=v'$ if and only if $\omega=\omega'$, whenever $x=x'$.

For any matrix-valued function $M:\RR^{2d}\longrightarrow\RR^{d\times d}$ and any vector field $V:\RR^{2d}\longrightarrow\RR^d$ sufficiently integrable, we define the convolution $[M*V]:\RR^{2d}\longrightarrow\RR^d$ as
\begin{equation}\label{E-convolution-matrix-vector}
[M*V](x,\omega):=\int_{\RR^{2d}}M(x-x',\omega-\omega')\,V(x',\omega')\,dx'\,d\omega'.
\end{equation}
Similarly, for any other vector field $W:\RR^{2d}\longrightarrow \RR^d$ we define $V*[W]:\mathbb{R}^{2d}\longrightarrow\RR$ as
\begin{equation}\label{E-convolution-vector-vector}
V*[W](x,\omega):=\int_{\RR^{2d}}V(x-x',\omega-\omega')\cdot W(x',\omega')\,dx'\,d\omega'.
\end{equation}

Finally, we use the notation $A\lesssim B$ if there exists a constant $C>0$, independent of the relevant parameters, such that $A\leq CB$.  Moreover we use the short-hand notation $a'$ to emphasise whenever a function $a=a(x,v)$ depends on alternative variables $(x',v')$.

\section{Preliminaries}\label{sec:prelim}
The following section is dedicated to the presentation of the preliminary information and results: rigorous definition of the first-to-second order transformation, weak formulations and basic properties of the solutions.

\subsection{First-to-second order transformation}
Our goal is to derive $\boldsymbol{f}$ solving \eqref{sysk} as a transformation of $\boldsymbol{\mu}$ solving \eqref{kurak} and vice versa. While the transformation introduced in \eqref{trans} is well defined if $\boldsymbol{\mu}$ and $\boldsymbol{f}$ are smooth enough, we need a similar notion applicable for measures. 

\begin{defi}[Change of variables]\label{D-change-variables}
Let $\rho\in \mathcal{P}_2(\mathbb{R}^d)$ be any probability measure. Then, we define the following transformations $\mathcal{T}^{1\to 2}[\rho],\mathcal{T}^{2\to 1}[\rho]:\mathbb{R}^{2d}\longrightarrow\mathbb{R}^{2d}$ of the phase space:
\begin{align*}
\begin{aligned}
\mathcal{T}^{1\to 2}[\rho](x,\omega)&:=(x,\boldsymbol{u}[\rho](x,\omega)), && (x,\omega)\in \mathbb{R}^{2d},\\
\mathcal{T}^{2\to 1}[\rho](x,v)&:=(x,\boldsymbol{\omega}[\rho](x,v)), && (x,v)\in \mathbb{R}^{2d},
\end{aligned}
\end{align*}
where
\begin{equation}\label{uomega}
    \begin{split}
          \boldsymbol{u}[\rho](x,\omega) = \omega-\nabla W*\rho(x),\\
    \boldsymbol{\omega}[\rho](x,v) = v+\nabla W*\rho(x).
    \end{split}
\end{equation}
In addition, given any $\mu,f\in \mathcal{P}(\mathbb{R}^{2d})$ with $\pi_{x\#}\mu=\rho=\pi_{x\#}f$, we define
\begin{align}
f^\mu&:=\mathcal{T}^{1\to 2}[\rho]_{\#}\mu\in \mathcal{P}(\mathbb{R}^{2d}),\label{defiT}\\
\mu^f&:=\mathcal{T}^{2\to 1}[\rho]_{\#}f\in \mathcal{P}(\mathbb{R}^{2d})\label{defiTstar}.
\end{align}
\end{defi}

We note that the above transformations $\mathcal{T}^{1\to 2}[\rho]$ and $\mathcal{T}^{2\to 1}[\rho]$ are well defined by the assumption $\rho\in \mathcal{P}_2(\mathbb{R}^d)$ and the growth $\vert \nabla W(x)\vert\lesssim\vert x\vert^{1-\alpha}$. In addition, they are inverse to each other by definition. Therefore, the pushforward map $\mathcal{T}^{1\to 2}[\rho]_{\#}$ in \eqref{defiT} defines a bijective mapping over the subspace of probability measures $\mathcal{P}(\mathbb{R}^{2d})$ with fixed marginal $\rho$ with respect to $x$. In addition, its inverse is given by the pushforward map $\mathcal{T}^{2\to 1}[\rho]_{\#}$ in \eqref{defiTstar}. Moreover, we remark that the spacial distribution stays unchanged under those transformations since we have
$$\pi_{x\#} f^\mu=\rho=\pi_{x\#} \mu^f.$$
Finally, we note that if $\mu$ and $f$ are densities (and since $\mathcal{T}^{1\to 2}[\rho]$ and $\mathcal{T}^{2\to 1}[\rho]$ are measure-preserving maps), then $f^\mu$ and $\mu^f$ in \eqref{defiT} and \eqref{defiTstar} simply take the form
\begin{align*}
f^\mu(x,v)&=\mu(x,v+\nabla W*\rho(x)),\\
\mu^f (x,\omega)&=f(x,\omega-\nabla W*\rho(x)).
\end{align*}

\subsection{Weak formulations for \eqref{sysk} and \eqref{kurak}}
Next, we proceed  by introducing the appropriate weak formulations for $\boldsymbol{f}$ and $\boldsymbol{\mu}$.
The starting point is the classical notion of measure-valued solution for the alignment system \eqref{sysk} associated with the identity
\begin{align}\label{DCSprep}
\int_0^T\int_{\RR^{2d}} \Bigg(\partial_t\psi_t + v\cdot\nabla_x\psi_t - \Big((D^2 W v)*f_t\Big)\cdot\nabla_v\psi_t\Bigg) \,df_t\,dt =
- \int_{\RR^{2d}} \psi_0\,df_0,
\end{align}
 which holds for any smooth test function $\psi=\psi_t(x,v)$, compactly supported in $[0,T)$. However, note that for singular $\phi$ as in \eqref{ker}, the nonlinear term may be ill-defined since $|D^2W|\approx \phi$. We circumvent this problem by diagonalizing such a term thanks to the anti-symmetry of $D^2 W(x-x')(v-v')$ under the change of variables $(x,v)\leftrightarrow (x',v')$. Formally we have
\begin{align}\label{sym}
\begin{aligned}
\int_0^T\int_{\RR^{2d}} &\Big((D^2 W v)*f\Big)\cdot\nabla_v\psi_t(x,v)\,df_t \,dt\\
&= \int_0^T\int_{\RR^{4d}} \nabla_v\psi_t(x,v)^\top D^2 W(x-x')(v-v') \,d(f_t\otimes f_t') \,dt\\ 
&= -\frac{1}{2} \int_0^T\int_{\RR^{4d}}(\nabla_v\psi_t(x,v)-\nabla_v\psi_t(x',v'))^\top D^2 W(x-x') (v'-v) \,d(f_t\otimes f_t') \,dt\\
&= -\frac{1}{2} \frac{1}{1-\alpha}\int_0^T\int_{\RR^{4d}}(\nabla_v\psi_t(x,v)-\nabla_v\psi_t(x',v'))\cdot (v'-v)\,\phi(|x-x'|)\,d(f_t\otimes f_t')\,dt\\
 &\qquad +\frac{1}{2}\frac{\alpha}{1-\alpha}\int_0^T\int_{\RR^{4d}}(\nabla_v\psi_t(x,v)-\nabla_v\psi_t(x',v'))\cdot(x'-x)\\
 &\qquad\qquad\times \phi(|x-x'|)\frac{(x-x')\cdot (v-v')}{|x-x'|^2}\,d(f_t\otimes f_t') \,dt,
\end{aligned}
\end{align}
where the last line contains the expansion of $D^2W$ for the readers' convenience. We note that in doing so, the factor $\nabla_v\psi_t(x,v)-\nabla_v\psi_t(x',v')$ coming from the diagonalization cancels one order of the singularity of $\phi(|x-x'|)$ with respect to $|x-x'|$ on the right-hand side of \eqref{sym} thanks to the Lipschitz-continuity of $\nabla_v\psi$, see Remark \ref{welldefi} below. More precisely, we remark that the right-hand side of \eqref{sym} can rigorously be defined only whenever $x\neq x'$. Surprisingly, as we prove in Lemma \ref{enstrophy}, the collision set {\rm Col} defined in \eqref{colisionset} is negligible with respect to the measure $f_t\otimes f_t\otimes \mathcal{L}^1_{\mathlarger{\llcorner}[0,T]}(t)$, where $\mathcal{L}^1_{\mathlarger{\llcorner}[0,T]}$ stands for the 1D Lebesgue measure restricted to $[0,T]$. Therefore, the only problematic situation is reduced to the diagonal set $\Delta$ defined in \eqref{diagonalset}. To make the above calculations rigorous, we remove the diagonal set from all integrals since it has no actual contribution in accordance with the particle system, see \cite{P-14} and also \cite{PPS-21}.

Based on the above computation we introduce the weak formulation for \eqref{sysk}.

\begin{defi}[Weak formulation for \eqref{sysk}]\label{DCS}
Consider any $T>0$ and $f_0\in {\mathcal P}_2(\mathbb{R}^{2d})$. We say that $\boldsymbol{f}\in C([0,T], {\mathcal P}(\mathbb{R}^{2d})-\mbox{narrow})$  is a weak measure-valued solution to \eqref{sysk} in the time interval $[0,T]$, subject to the initial datum $f_0$, if the following weak formulation is verified
\begin{align}\label{weakeqf}
- \int_{\RR^{2d}} &\psi_0\,df_0 = \int_0^T\int_{\RR^{2d}} \left(\partial_t\psi_t + v\cdot\nabla_x\psi_t\right)\,df_t\,dt\\
&+ \frac{1}{2} \int_0^T\int_{\RR^{4d}\setminus\Delta}(\nabla_v\psi_t(x,v)-\nabla_v\psi_t(x',v'))^\top D^2 W(x-x') (v'-v) \,d(f_t\otimes f_t')\,dt,\nonumber
\end{align}
 for all test functions $\psi\in C^1_c([0,T)\times\RR^{2d})$, compactly supported in $[0,T)$ with Lipschitz-continuous velocity gradient $\nabla_v\psi_t$ (uniformly with respect to $t$). Furthermore, the kinetic energy and enstrophy satisfy the inequality
\begin{equation}\label{weakeqfa}
\int_{\RR^{2d}}|v|^2\,df_T + \int_0^T \int_{\RR^{4d}\setminus\Delta}|v-v'|^2\phi(|x-x'|)\,d(f_s\otimes f_s') \,ds \leq \int_{\RR^{2d}}|v|^2\,df_0.
\end{equation}
Finally, the second-order moments with respect to $x$ are uniformly bounded
\begin{equation}\label{weakeqfb}
\sup_{t\in[0,T]}\int_{\RR^{2d}}|x|^{2}\,df_t <\infty.
\end{equation}
\end{defi}

\begin{rem}[Good definition]\label{welldefi}
Consider $\psi\in C^1([0,T]\times \mathbb{R}^d)$ with Lipschitz-continuous $\nabla_v\psi$ (uniformly with respect to $t$). Then, Young's inequality and Lemma \ref{szwarc} (iii) imply that
\begin{align}\label{E-welldefi}
\begin{aligned}
&\left\vert (\nabla_v\psi_t(x,v)-\nabla_v\psi_t(x',v'))^\top D^2 W(x-x')\,(v'-v)\right\vert\\
&\qquad \qquad \qquad\lesssim [\nabla_v\psi]_{\rm Lip}\,(|x-x'|+|v-v'|)\,\phi(|x-x'|)\,|v-v'|\\
&\qquad \qquad \qquad\lesssim [\nabla_v\psi]_{\rm Lip}\left(|v-v'|^2\phi(|x-x'|)+|x-x'|^2\phi(|x-x'|)\right)\\
&\qquad \qquad \qquad\lesssim [\nabla_v\psi]_{\rm Lip}\left(|v-v'|^2\phi(|x-x'|)+|x|^{2}+|x'|^{2}+1\right),
\end{aligned}
\end{align}
for any $t\in [0,T]$ and $((x,v),(x',v'))\in \mathbb{R}^{4d}\setminus \Delta$. Above, we have used the identity $\phi(r)r^2=r^{2-\alpha}$, and we have interpolated $(2-\alpha)$-order moments by second-order moments. Since the right hand side in \eqref{E-welldefi} is integrable with respect to $f_t\otimes f_t\otimes \mathcal{L}^1_{\mathlarger{\llcorner}[0,T]}(t)$ on $(\mathbb{R}^{4d}\setminus \Delta)\times [0,T]$ thanks to the properties \eqref{weakeqfa} and \eqref{weakeqfb}, then we note that each term in \eqref{weakeqf} is well-defined.
\end{rem}

\begin{rem}
Inequalities \eqref{weakeqfa} and \eqref{weakeqfb} are standard in {\rm non-singular} Cucker-Smale-type alignment models and can be directly derived from \eqref{weakeqf} under suitable uniform-in-time tightness assumptions on the solution $\boldsymbol{f}$. For the non-singular CS model, it was shown in \cite{HL-09}, that if $f_0$ is compactly supported then the support of $\boldsymbol{f}$ is uniformly bounded in $[0,T]$. Thus, the second moments in both $x$ and $v$ are uniformly bounded, from which  \eqref{weakeqfa} and \eqref{weakeqfb} readily follow. The singular variant of the CS model is an entirely different story, wherein condition \eqref{weakeqfa} plays the role of an {\it a priori} assumption ensuring good definition of the last integral in \eqref{weakeqf}. In singular models, inequality \eqref{weakeqfa} can be derived on the level of mean-field approximation and then carried over to solutions of kinetic equations, as was done in \cite{MP-18}. However, to the best of our knowledge, in singular alignment models, \eqref{weakeqfa} is yet to be deduced solely based on the weak formulation \eqref{weakeqf}. Naturally, this issue can be boiled down to the question of uniqueness of solutions. In the present paper we construct solutions satisfying both \eqref{weakeqfa} and \eqref{weakeqfb} and they are unique provided that \eqref{nudiag0} holds too.
\end{rem}

Under the assumptions in Definition \ref{DCS}, we obtain improved properties on any weak measure-valued solution to \eqref{sysk}, which we summarize in the following result.

\begin{pro}\label{W2contf}
Consider any $T>0$ and $f_0\in \mathcal{P}_2(\mathbb{R}^{2d})$, and let $\boldsymbol{f}$ be any weak measure-valued solution to \eqref{sysk} with initial datum $f_0$ in the sense of Definition \ref{DCS}. Then, we have
\begin{enumerate}[label=(\roman*)]
\item (Second-order moments) The following uniform bound holds true:
$$\sup_{t\in [0,T]}\int_{\mathbb{R}^{2d}}|x|^2+|v|^2\,df_t<\infty.$$
\item (Time-continuity) The solution has improved time-continuity:
$$\boldsymbol{f}\in C([0,T],\mathcal{P}_2(\mathbb{R}^{2d})-W_2).$$
\item (Admissible test functions) The weak formulation in Definition \ref{DCS} holds for general test functions with non-compact support, namely \eqref{weakeqf} remains true for any $\psi\in C^1_b([0,T]\times \mathbb{R}^{2d})$, compactly supported in $[0,T)$, with Lipschitz-continuous $\nabla_v\psi$ (uniformly in $t$).
\end{enumerate}
\end{pro}

\begin{proof}
We start by claiming that the following two identities hold true
\begin{align}
\int_{\RR^{2d}}|v|^2\, df_t = \int_{\RR^{2d}}|v|^2\,df_s &- \int_{s}^t\int_{\RR^{4d}\setminus \Delta} (v-v')^\top D^2 W(x-x') (v-v')\ d(f_r \otimes f_r) \ dr,\label{E-second-moments-v}\\
\int_{\RR^{2d}}|x|^2\, df_t = \int_{\RR^{2d}}|x|^2\, df_s &- \int_{s}^t\int_{\RR^{4d}\setminus \Delta} (x-x')^\top D^2 W(x-x') (v-v')\ d(f_r \otimes f_r) \ dr
\label{E-second-moments-x}\\
&\hspace{5cm} + 2\int_s^t \int_{\RR^{2d}} x\cdot v\ df_r\ dr,\nonumber
\end{align}
for any $0\leq s\leq t\leq T$. We shall prove the above claim later, but first we use it to prove the above items $(i)$, $(ii)$ and $(iii)$.

\medskip

$\diamond$ {\sc Step 1}. Proof of $(i)$ and $(ii)$.\\
On the one hand, the uniform bound of the second-order moments with respect to $x$ in $(i)$ follows from the hypothesis \eqref{weakeqfb} in Definition \ref{DCS}. In addition, setting $s=0$ in \eqref{E-second-moments-v} and using Lemma \ref{szwarc}(ii) on the second integral in the right-hand side, we infer that the velocity second-order moments with respect to $v$ are non-increasing and therefore
$$\int_{\mathbb{R}^{2d}}|v|^2\,df_t\leq \int_{\mathbb{R}^{2d}}|v|^2 \,df_0,$$
for any $t\in [0,T]$. Since $f_0\in \mathcal{P}_2(\mathbb{R}^{2d})$, then we also have the uniform bound of the second-order moments with respect to $v$ in $(i)$, and in particular they are finite for every $t\in [0,T]$.

On the other hand, let us define the following couple of functions:
\begin{align*}
g(t)&:=-\int_{\RR^{4d}\setminus \Delta} (v-v')^\top D^2 W(x-x') (v-v')\ d(f_t \otimes f_t),\\
h(t)&:=-\int_{\RR^{4d}\setminus \Delta} (x-x')^\top D^2 W(x-x') (v-v')\ d(f_t \otimes f_t)+2\int_{\mathbb{R}^{2d}}x\cdot v\,df_t,
\end{align*}
for any $t\in [0,T]$. By Remark \ref{welldefi}, properties \eqref{weakeqfa}-\eqref{weakeqfb}, and the above uniform bound of second-order moments in $(i)$ the functions $g,h$ are well defined and belong to $L^1(0,T)$. Therefore, \eqref{E-second-moments-v}-\eqref{E-second-moments-x} infer that the following two functions are (absolutely) continuous
$$t\in [0,T]\longmapsto \int_{\mathbb{R}^{2d}} |x|^2\,df_t\quad\mbox{and}\quad t\in [0,T]\longmapsto \int_{\mathbb{R}^{2d}} |v|^2\,df_t.$$
Note that the above readily implies $(ii)$. Specifically, since $\boldsymbol{f}$ is time-continuous in the narrow topology by hypothesis, then we have that $f_{t_n}\rightarrow f_{t_*}$ narrowly for any $\{t_n\}_{n\in \mathbb{N}}\subset [0,T]$ and $t_*\in [0,T]$ such that $t_n\rightarrow t_*$. In addition, by the time-continuity of moments above, we have
$$\int_{\mathbb{R}^{2d}}|x|^2+|v|^2\,df_{t_n}\rightarrow \int_{\mathbb{R}^{2d}}|x|^2+|v|^2\,df_{t_*}.$$
This of course implies that $f_{t_n}\rightarrow f_{t_*}$ with respect to the $W_2$ Wasserstein distance by the usual characterization, see \cite[Theorem 6.9]{V-09}.

\medskip

$\diamond$ {\sc Step 2}. Proof of $(iii)$.\\
We employ a classical cut-off argument, which exploits the above uniform bound of second-order moments. To such an end, we set any $\psi\in C^1_b([0,T]\times \mathbb{R}^{2d})$, compactly supported in $[0,T)$, with Lipschitz-continuous $\nabla_v \psi$ (uniformly in $t$). Let us consider a smooth cut-off function $\chi\in C^\infty_c(\mathbb{R}^{2d})$, $0\leq \chi\leq 1$, with $\chi=0$ for $|(x,v)|\geq 2$ and $\chi=1$ for $|(x,v)|\leq 1$, and we define
$$\chi_R(x,v)=\chi\left(\frac{x}{R},\frac{v}{R}\right),\quad (x,v)\in \mathbb{R}^{2d},$$
for any $R>0$. Then, we can truncate $\psi$ and define $\psi^R_t(x,v):=\chi_R(x,v)\,\psi_t(x,v)$ which now belongs to $C^1_c([0,T)\times \mathbb{R}^{2d})$ and verifies that $\nabla_v\psi_R$ is again Lipschitz-continuous (uniformly in $t$). Since $\psi_R$ satisfies the appropriate assumptions, then the weak formulation \eqref{weakeqf} holds
\begin{align}\label{E-weakeqf-truncations}
- \int_{\RR^{2d}} &\psi_{0}^R\,df_0 = \int_0^T\int_{\RR^{2d}} \left(\partial_t\psi_t^R + v\cdot\nabla_x\psi_t^R\right)\,df_t\,dt\\
&+ \frac{1}{2} \int_0^T\int_{\RR^{4d}\setminus \Delta}(\nabla_v\psi_t^R(x,v)-\nabla_v\psi_t^R(x',v'))^\top D^2 W(x-x') (v'-v) \,d(f_t\otimes f_t')\,dt,\nonumber
\end{align}
Our final goal is to show that we can pass to the limit as $R\rightarrow \infty$ and therefore the weak formulation \eqref{weakeqf} also holds for the test function $\psi$. We remark that
$$1-\chi_R(x,v)\leq \mathds{1}_{\mathbb{R}^{2d}\setminus B_R}(x,v)\leq \frac{1}{R^2}(|x|^2+|v|^2),$$
where $\mathds{1}_{\mathbb{R}^{2d}\setminus B_R}$ is the characteristic function of the complement of the ball $B_R$ centered at $0$ with radius $R$. Therefore the first three terms in \eqref{E-weakeqf-truncations} can be controlled as follows
\begin{align*}
\left\vert\int_{\mathbb{R}^{2d}}(\psi_0^R-\psi_0)\,df_0\right\vert 
&\leq \frac{| \psi_0|_\infty}{R^2}\int_{\mathbb{R}^{2d}}|x|^2+|v|^2\,df_0,\\
\left\vert\int_0^T\int_{\mathbb{R}^{2d}}(\partial_t\psi^R-\partial_t\psi)\,df_t\,dt\right\vert 
&\leq \frac{T| \partial_t\psi|_\infty}{R^2}\sup_{t\in [0,T]}\int_{\mathbb{R}^{2d}} |x|^2+|v|^2\,df_t,\\
\left\vert\int_0^T\int_{\mathbb{R}^{2d}} v\cdot (\nabla_x\psi^R-\nabla_x\psi)\,df_t\,dt\right\vert&\lesssim \frac{T| \nabla_x\psi|_\infty}{R}\sup_{t\in [0,T]}\int_{\mathbb{R}^{2d}}|x|^2+|v|^2\,df_t,
\end{align*}
which converge to zero as $R\rightarrow \infty$ thanks to (i). For the nonlinear term in \eqref{E-weakeqf-truncations}, we note that
\begin{multline*}
\left\vert(\nabla_v\psi_t^R(x,v)-\nabla_v\psi_t(x,v))-(\nabla_v\psi_t^R(x',v')-\nabla_v\psi_t(x',v'))\right\vert\\
\leq C\left((1-\chi_R(x,v))+\frac{1}{R}\right)(|x-x'|+|v-v'|)\\
\leq C\left(\mathds{1}_{\mathbb{R}^{2d}\setminus B_R}(x,v)+\frac{1}{R}\right)(|x-x'|+|v-v'|),
\end{multline*}
for any $t\in [0,T]$, each $(x,v),(x',v')\in \mathbb{R}^{2d}$ and some $C>0$ depending on the uniform and Lipschitz norms of $\psi$ and $\chi$ and their first order derivatives. Arguing as in Remark \ref{welldefi} we have
\begin{align*}
&\left\vert\int_0^T\int_{\mathbb{R}^{4d}\setminus \Delta} \left((\nabla_v\psi^R-\nabla_v\psi)-(\nabla_v\psi^{R}\,'-\nabla_v\psi\,')\right)^\top D^2 W(x-x')(v'-v)\,d(f_t\otimes df_t)\,dt\right\vert\\
&\quad \lesssim\int_0^T \int_{\mathbb{R}^{4d}\setminus \Delta} \left(\mathds{1}_{\mathbb{R}^{2d}\setminus B_R}(x,v)+\frac{1}{R}\right)\big(|v-v'|^2\phi(|x-x'|)+|x|^{2}+|x'|^{2}+1\big)\,d(f_t\otimes f_t)\,dt,
\end{align*}
which again vanishes as $R\rightarrow \infty$ by the properties \eqref{weakeqfa}-\eqref{weakeqfb}, and we end this part.

\medskip

$\diamond$ {\sc Step 3}. Proof of the claim.\\
We note that the main difficulty is that we cannot take $|x|^2$ or $|v|^2$ as test functions in the weak formulation \eqref{weakeqf} because they do not have compact support. Our proof is then based on a cut-off argument again. For any $\varphi\in C^1_b(\mathbb{R}^{2d})$ we define the truncation 
$$\varphi^R(x,v):=\bar\chi_R(x,v)\,\varphi(x,v),\quad (x,v)\in \mathbb{R}^2,$$
with a similar cut-off function $\bar\chi_R$ as above, and for any arbitrary $\eta\in C^1_c(0,T)$ we can also set
$$\psi_t^R(x,v):=\eta(t)\,\varphi^R(x,v),\quad t\in [0,T],\,(x,v)\in \mathbb{R}^{2d}.$$
Note that we have $\psi^R\in C^1_c([0,T)\times \mathbb{R}^{2d})$ and $\nabla_v\psi_t^R$ is Lipschitz-continuous (uniformly in $t$), and therefore the weak formulation \eqref{weakeqf} implies
\begin{equation}\label{E-g-h-weak-derivative}
\int_0^T \eta'(t) g(t)\,dt=-\int_0^T \eta(t) h(t)\,dt,
\end{equation}
for every $\eta\in C^1_c(0,T)$, where $g,h\in [0,T]\longrightarrow \mathbb{R}$ are the functions given by
\begin{align*}
g(t)&:=\int_{\mathbb{R}^{2d}}\varphi^R(x,v)\,df_t(x,v),\\
h(t)&:=\int_{\mathbb{R}^{4d}\setminus \Delta} \left(\nabla_v\varphi^R(x,v)-\nabla_v\varphi^R(x',v')\right)^\top D^2 W(x-x')(v'-v)\,d(f_t\otimes f_t)+\int_{\mathbb{R}^{2d}}v\cdot \nabla_x\varphi^R\,df_t.
\end{align*}
Note that $g\in L^1(0,T)$ (indeed $L^\infty(0,T)$) and also arguing like in Remark \ref{welldefi} we have
$$\vert h(t)\vert\lesssim [\nabla _v\varphi^R]_{\rm Lip}\int_{\mathbb{R}^{4d}\setminus \Delta} (|v-v'|^2\phi(x-x')+|x|^2+|x'|^2+1)\,d(f_t\otimes f_t)+R\,| \nabla_x\varphi^R|_\infty,$$
for {\it a.e.} $t\in [0,T]$. Since the right hand side is an integrable function by \eqref{weakeqfa}, then $h\in L^1(0,T)$. By \eqref{E-g-h-weak-derivative} it follows that $g'=h$ in distributional sense, and hence $g\in W^{1,1}(0,T)$. Therefore, we can apply the fundamental theorem of calculus to obtain
$$g(t)-g(s)=\int_s^t h(r)\,dr,$$
for {\it a.e.} $t,s\in [0,T]$. Since $g$ is continuous because $\boldsymbol{f}\in C([0,T],\mathcal{P}_2(\mathbb{R}^{2d})-narrow)$ and $\varphi^R$ are bounded, then the above must actually hold for every $t,s\in [0,T]$, that is,
\begin{align}\label{E-weakeqf-truncations-integrated}
\begin{aligned}
\int_{\mathbb{R}^{4d}}\varphi^R&\,df_t-\int_{\mathbb{R}^{4d}}\varphi^R\,df_s\\
&=\int_s^t \int_{\mathbb{R}^{4d}\setminus \Delta} \left(\nabla_v\varphi^R(x,v)-\nabla_v\varphi^R(x',v')\right)^\top D^2 W(x-x')(v'-v)\,d(f_r\otimes f_r)\,dr\\
&\qquad \qquad +\int_s^t \int_{\mathbb{R}^{2d}}v\cdot \nabla_x\varphi^R\,df_r\,dr,
\end{aligned}
\end{align}
for every $t,s\in [0,T]$. Taking $\varphi(x,v)=|v|^2$ and $\bar \chi_R(x,v)=\chi\left(\frac{x}{R^4},\frac{v}{R}\right)$ note that we can take limits as $R\rightarrow \infty$ in \eqref{E-weakeqf-truncations-integrated} and obtain \eqref{E-second-moments-v}. Specifically, we use the monotone convergence theorem on the left-hand side. We argue as in the previous step for the non-linear term, and we observe that by our choice of $\bar\chi_R$ the last term can be bounded by
$$\left\vert\int_s^t \int_{\mathbb{R}^{2d}} v\cdot \nabla_x\varphi^R\,df_r\,dr\right\vert=\left\vert\int_s^t\int_{\mathbb{R}^{2s}}v |v|^2\nabla_x\bar\chi_R(x,v)\,df_r\,dr\right\vert\leq  \frac{T |\nabla_x\chi|_\infty}{R},$$
which vanishes as $R\rightarrow \infty$. With \eqref{E-second-moments-v} in hand, we have $(i)$ and we proceed in an analogous way to prove \eqref{E-second-moments-x} by passing to the limit in \eqref{E-weakeqf-truncations-integrated} for $\varphi(x,v)=|x|^2$ and $\bar\chi_R=\chi\left(\frac{x}{R^2},\frac{v}{R}\right)$. All the terms can be handled analogously except for the last one, for which we have
\begin{align*}
\left\vert v\cdot\nabla_x\varphi^R(x,v)-2x\cdot v\right\vert&=2|x||v|(1-\bar\chi_R(x,v))+|x|^2|v||\nabla_x\bar\chi_R(x,v)|\\
&\leq (1-\bar \chi_R(x,v))(|x|^2+|v|^2)+\frac{|\nabla_x\chi|_\infty}{R}|x|^2,
\end{align*}
for any $(x,v)\in \mathbb{R}^{2d}$. Hence, we obtain
$$\left\vert\int_s^t \int_{\mathbb{R}^{2d}}v\cdot (\nabla_x\varphi^R-2x)\,df_r\,dr\right\vert\lesssim  \int_0^T\int_{\mathbb{R}^{2d}}\left(\mathds{1}_{\mathbb{R}^{2d}\setminus B_R}(x,v)+\frac{1}{R}\right) (|x|^2+|v|^2)\,df_r\,dr,$$
which vanishes with $R\rightarrow \infty$ thanks to $(i)$.
\end{proof}

We proceed with the weak formulation for the first order system \eqref{kurak}. The goal is to provide a notion of solutions that is compatible with \cite[Theorem 5.3]{PP-22-1-arxiv} and, at the same time, which can be obtained from Definition \ref{DCS} by the change of variables from Definition \ref{D-change-variables}.

\begin{defi}[Weak formulation for \eqref{kurak}]\label{D1st}
Consider any $T>0$ and $\mu_0\in{\mathcal P}_2({\mathbb{R}^{2d})}$. We say that  $\boldsymbol{\mu}\in C([0,T],\mathcal{P}(\mathbb{R}^{2d})-\mbox{narrow})$, satisfying the assumption
\begin{equation}\label{E-weakeqF}
\int_0^T\int_{\mathbb{R}^{2d}}|x|^{1-\alpha}\,d\mu_t(x,\omega)\,dt<\infty,
\end{equation}
is a weak measure-valued solution to \eqref{kurak} in the time interval $[0,T]$, subject to initial datum $\mu_0$, if the following weak formulation is verified
\begin{align}\label{weakeqF}
- \int_{\mathbb{R}^{2d}} \eta_0\,d\mu_0 = \int_0^T\int_{\mathbb{R}^{2d}}\Big(\partial_t\eta_t(x,\omega) + (\omega-\nabla W*\rho_t(x))\cdot\nabla_x\eta_t(x,\omega)\Big) \,d\mu_t \,dt,
\end{align}
for all test functions $\eta \in C^1_b([0,T]\times\RR^d)$, compactly supported in $[0,T)$.
\end{defi}

\begin{rem}[Good definition]
Note that, unlike Definition \ref{DCS}, Definition \ref{D1st} does not require a diagonalization akin to \eqref{sym}, since the problematic singular term does not appear. In fact, it is hidden behind the change of variables in Definition \ref{D-change-variables}. Furthermore, equivalents of conditions \eqref{weakeqfa}-\eqref{weakeqfb} do not appear in Definition \ref{D1st} as they are not required to ensure that all the terms in \eqref{weakeqF} are well defined. Indeed, condition \eqref{E-weakeqF} is enough since it guarantees through Corollary \ref{C-sided-Lipschitz}(iii) that $\boldsymbol{u}[\rho]\in L^1(0,T;L^1_{\mu_t}(\mathbb{R}^{2d},\mathbb{R}^d))$ is well defined. In Proposition \ref{W2cont} and Lemma \ref{enstrophy} we prove in fact that \eqref{weakeqfa}-\eqref{weakeqfb} are a consequence of \eqref{E-weakeqF}-\eqref{weakeqF}.
\end{rem}

\begin{pro}\label{W2cont}
Consider any $T>0$ and $\mu_0\in \mathcal{P}_2(\mathbb{R}^{2d})$, and let $\boldsymbol{\mu}$ be any weak measure-valued solution to \eqref{kurak} with initial datum $\mu_0$ in the sense of Definition \ref{D1st}. Then, we have
\begin{enumerate}[label=(\roman*)]
\item (Second-order moments) The following uniform bound holds true:
\begin{equation}\label{kurakmoment1}
    \sup_{t\in[0,T]}\int_{\RR^{2d}}|x|^2 + |\omega|^2\,d\mu_t(x,\omega) < +\infty.
\end{equation}
\item (Time-continuity) The solution has improved time-continuity:
$$\boldsymbol{\mu}\in C([0,T],\mathcal{P}_2(\mathbb{R}^{2d})-W_2).$$
\end{enumerate}
\end{pro}

\begin{proof}

$\diamond$ {\sc Step 1}. Proof of $(i)$.\\
By virtue of Corollary \ref{C-Lagrangian-solution}, we infer that $\boldsymbol{\mu}$ is a Lagrangian solution. Thus $\mu_t=Z(t;\cdot,\cdot)_{\#}\mu_0$ for all $t\in [0,T]$, where $Z=Z(t;x,\omega)$ is the forward-unique characteristic flow solving \eqref{E-characteristic-system} in the sense of Carath\'eodory. Therefore, we obtain
\begin{equation}\label{E-moments-push-forward}
\int_{\mathbb{R}^{2d}}|x|^2+|\omega|^2\,d\mu_t(x,\omega)=\int_{\mathbb{R}^{2d}}|X(t;x,\omega)|^2\,d\mu_0(x,\omega)+\int_{\mathbb{R}^{2d}}|\omega|^2\,d\mu_0(x,\omega),
\end{equation}
for all $t\in [0,T]$. Thereby, all that we need to do is to control the expansion of the characteristic trajectories \eqref{E-characteristic-system} in the first term above. By the fundamental theorem of calculus (which holds because Carath\'eodory solutions are absolutely continuous) and the sublinear growth of the velocity field in Corollary \ref{C-sided-Lipschitz}(iii) we obtain the following control
\begin{align*}
|X(t;x,\omega)-x|&\leq \int_0^t |\boldsymbol{u}[\rho_s](X(s;x,\omega),\omega)|\,ds\\
&\leq \frac{M}{1-\alpha}+T\,|\omega|+\frac{1}{1-\alpha}\int_0^t|X(s;x,\omega)|^{1-\alpha}\,ds\\
&\leq \frac{M}{1-\alpha}+\frac{\alpha}{1-\alpha}+T\,(|x|+|\omega|)+\int_0^t|X(s;x,\omega)-x|\,dx,
\end{align*}
for all $t\in [0,T]$, where $M:=\int_0^T\int_{\mathbb{R}^{2d}}|x|^{1-\alpha}\,d\mu_t(x,\omega)\,dt$, and we have interpolated lower-order moments by first-order moments. Therefore, Gr\"{o}nwall's lemma implies that
$$|X(t;x,\omega)|\leq C_T \,(1+|x|^2+|\omega|^2)^{1/2},$$
for all $t\in [0,T]$, and some $C_T>0$. Using the above control on \eqref{E-moments-push-forward} yields
$$\int_{\mathbb{R}^{2d}}|x|^2+|\omega|^2\,d\mu_t(x,\omega)\leq (1+C_T^2)\int_{\mathbb{R}^{2d}}|x|^2+|\omega|^2\,d\mu_0(x,\omega),$$
and this ends the proof of $(i)$.

\medskip

$\diamond$ {\sc Step 2}. Proof of $(ii)$.\\
By the propagation of moments above, we have 
\begin{equation}\label{W2cont-1}
    \sup_{t\in [0,T]}\int_{\mathbb{R}^{2d}}|\boldsymbol{u}[\rho_t](x,\omega)|^2\,d\mu_t(x,\omega)<\infty
\end{equation}
and therefore $\boldsymbol{\mu}\in {\rm Lip}(0,T;\mathcal{P}_{2,\nu}(\mathbb{R}^{2d}))$, where $({\mathcal P}_{2,\nu}(\mathbb{R}^{2d}), W_{2,\nu})$ with $\nu:=\pi_{\omega\#}\mu_0$ is the {\it fibered Wasserstein space} introduced in \cite{PP-22-1-arxiv}, and also exploited in \cite{P-19-arxiv} and \cite{MP-19-arxiv}. In particular, since we have the embedding $\mathcal{P}_{2,\nu}(\mathbb{R}^{2d})\subset \mathcal{P}_2(\mathbb{R}^{2d})$, then $\boldsymbol{\mu}\in C(0,T;\mathcal{P}_2(\mathbb{R}^{2d})-W_2)$ and specifically
$$W_2(\mu_{t_1},\mu_{t_2})\leq W_{2,\nu}(\mu_{t_1},\mu_{t_2})\leq |t_1-t_2|\,\sup_{t\in [0,T]}\int_{\mathbb{R}^{2d}}\vert \boldsymbol{u}_t(x,\omega)\vert^2\,d\mu_t(x,\omega),$$
for all $t_1,t_2\in [0,T]$.
\end{proof}

\begin{cor}[Fibered gradient flows]\label{C-measure-sol-gradient-flows}
Weak measure-valued solutions to \eqref{kurak} in the sense of Definition \ref{D1st} exist, are unique and they coincide with gradient flows of the energy functional
$$\mathcal{E}[\mu]:=-\int_{\mathbb{R}^{2d}}x\cdot \omega\,d\mu(x,\omega)+\int_{\mathbb{R}^{4d}}W(x-x')\,d\mu(x,\omega)\,d\mu(x',\omega'),\quad \mu\in \mathcal{P}_{2,\nu}(\mathbb{R}^{2d}),$$
with respect to the fibered Wasserstein space $(\mathcal{P}_{2,\nu}(\mathbb{R}^{2d}),W_{2,\nu})$, as established in \cite{PP-22-1-arxiv}.
\end{cor}

\begin{proof}
Since $\boldsymbol{\mu}$ verifies the weak formulation \eqref{weakeqF} in the distributional sense and we also have $\boldsymbol{\mu}\in AC(0,T;\mathcal{P}_{2,\nu}(\mathbb{R}^{2d}))$ by Proposition \ref{W2cont}(ii), then $\boldsymbol{\mu}$ amounts to a distributional solution in the sense of \cite[Definition 5.1]{PP-22-1-arxiv}. Using the machinery in that paper we are able to identify
$$\boldsymbol{u}[\rho_t]=-\nabla_{W_{2,\nu}}\mathcal{E}[\mu_t],\quad {\it a.e.}\ t\in [0,T],$$
and therefore weak measure-valued solutions in the sense of Definition \ref{D1st} simply coincide with fibered gradient flows of $\mathcal{E}$ in the sense of \cite[Definition 3.30]{PP-22-1-arxiv}. In fact, by virtue of the convexity properties of $W$ (thus also $\mathcal{E}$) in Appendix \ref{Appendix-convexity-W}, the machinery developed in \cite[Theorem 5.3]{PP-22-1-arxiv} applies, and therefore the existence and uniqueness of such fibered gradient flows follow.
\end{proof}

\section{Energy estimate for the first-order equation \eqref{kurak}}\label{sec:energest}

The goal of this section is to derive a first-order equivalent of the energy estimate \eqref{weakeqfa} from the weak formulation \eqref{weakeqF}. Note that we have already derived estimate \eqref{weakeqfb} in Proposition \ref{W2cont}. We aim to prove the following lemma.

\begin{lem}[Energy dissipation inequality]\label{enstrophy}
Consider any $T>0$ and $\mu_0\in \mathcal{P}_2(\mathbb{R}^{2d})$, and let $\boldsymbol{\mu}$ be the weak measure-valued solution to \eqref{kurak} with initial datum $\mu_0$ in the sense of Definition \ref{D1st}. Then, the following estimate holds true:
\begin{multline}\label{E-enstrophy}
\int_{\mathbb{R}^{2d}}\vert \boldsymbol{u}[\rho_T]\vert^2\,d\mu_T-\int_{\mathbb{R}^{2d}}\vert \boldsymbol{u}[\rho_0]\vert^2\,d\mu_0\\
+\int_0^T \int_{\mathbb{R}^{4d}\setminus \Delta}\vert\boldsymbol{u}[\rho_t](x,\omega)-\boldsymbol{u}[\rho_t](x',\omega') \vert^2\phi(|x-x'|)\,d(\mu_t \otimes\mu'_t)\,dt\leq 0,
\end{multline}
where, as usual, $\rho_t=\pi_{x\#}\mu_t$. In particular, the set ${\rm Col}$ of collisions of $\boldsymbol{\mu}$  defined in \eqref{colisionset} 
is negligible in the sense that $\lambda({\rm Col}) = 0$, where $\lambda = \mu_t\otimes\mu_t\otimes{\mathcal L}^1_{\mathlarger{\llcorner}[0,T]}(t)$ and ${\mathcal L}^1_{\mathlarger{\llcorner}[0,T]}$ denotes the 1D Lebesgue measure on the time interval $[0,T]$. 
\end{lem}

Before proving Lemma \ref{enstrophy}, let us establish the following shorthand notation for the {\it average kinetic energy} $E_k$ and the {\it dissipation of kinetic energy} $D_k$ (also called {\it enstrophy} in fluid mechanics) of any weak measure-valued solution $\boldsymbol{\mu}$ of \eqref{kurak}:
\begin{align*}
E_k[\mu_t] &:= \frac{1}{2}\int_{\RR^{2d}} \vert \boldsymbol{u}[\rho_t]\vert^2\,d\mu_t,\\
D_k[\mu_t] &:= \frac{1}{2}\int_{\RR^{4d}\setminus\Delta}\vert \boldsymbol{u}[\rho_t](x,\omega)-\boldsymbol{u}[\rho_t](x',\omega')\vert^2\phi(|x-x'|) \,d(\mu_t\otimes \mu_t'),
\end{align*}
see $\Delta$ in \eqref{diagonalset}.
\medskip

\begin{proof}[Proof of Lemma \ref{enstrophy}]

~

\medskip

The proof follows three steps: {\sc Step 1:} introduction of discrete versions of $E_k$ and $D_k$ for atomic measures, {\sc Step 2:} proof of \eqref{E-enstrophy} for atomic measures, {\sc Step 3:} proof of \eqref{E-enstrophy} for general measures via mean field limit, supplemented by uniqueness granted by  Corollary \ref{C-measure-sol-gradient-flows}.

\medskip

$\diamond\ $ {\sc Step 1}. The family of atomic solutions.

First, let us restrict our considerations to the special family of atomic solutions of \eqref{kurak} with $N$ atoms, {\it i.e.} solutions issued at atomic initial data with the shape of an empirical measure:
\begin{equation}\label{empirmu0}
\mu_0^N(x,\omega) = \frac{1}{N}\sum_{i=1}^N\delta_{x_{i0}^N}(x)\otimes \delta_{\omega_i^N}(\omega).    
\end{equation}
Consider the classical solution $(x_1^N(t),\ldots,x_N^N(t))$ of the Kuramoto-type particle system \eqref{kurap} with initial datum $(x_{1\,0}^N,\ldots,x_{N\,0}^N)$. Note that such a solution exists globally-in-time because the right hand side of the ODE system is continuous and has sub-linear growth. In addition, since $\nabla W$ is non-decreasing (thus, $-\nabla W$ is one-sided Lipschitz) by  Lemma \ref{szwarc}, then the solution is unique forwards-in-time by standard arguments (see also \cite{PPS-21} for similar results). Let us define the associated empirical measure by the relation
\begin{equation}\label{atom}
\mu^N_t(x,\omega) = \frac{1}{N}\sum_{i=1}^N\delta_{x_i^N(t)}(x)\otimes \delta_{\omega_i^N}(\omega).
\end{equation}
By the chain rule we readily infer that $\boldsymbol{\mu}^N$ above yields the unique weak-measure valued solution to \eqref{kurak} issued at $\mu^N_0$, {\it cf.} \eqref{empirmu0}. Consequently, in order to prove \eqref{E-enstrophy} for atomic solutions of the form \eqref{atom} it is sufficient to show its equivalent version for the particle system. To such an end, we denote the discrete versions of the kinetic energy and enstrophy as follows
\begin{align}\label{odetroph}
\begin{split}
E^N_k(t) &:= E_k[\mu_t^N] = \frac{1}{2N}\sum_{i=1}^N \vert u_i^N(t)\vert^2,\\
D^N_k(t) &:= D_k[\mu_t^N] = \frac{1}{2N^2}\sum_{i=1}^N\sum_{j\in\{1,...,N\}\setminus S_i^N(t)}^N|u_i^N(t)-u_j^N(t)|^2\phi(|x_i^N(t)-x_j^N(t)|),\\ 
S_i^N(t) &:= \{j\in\{1,...,N\}:\, x_i^N(t)=x_j^N(t),\quad \omega_i^N=\omega_j^N\},\\
u_i^N(t) &:= \dot{x}_i^N(t) \stackrel{\eqref{kurap}}{=} \omega_i^N + \frac{1}{N}\sum_{j=1}^N\nabla W(x_j^N(t)-x_i^N(t)).
\end{split}
\end{align}

\medskip

Note that while $E_k^N(t)$ is well defined, $D^N_k(t)$ may be ill defined whenever $x_i^N(t)=x_j^N(t)$ and $u_i^N(t)\neq u_j^N(t)$. This will be the main issue to overcome along the next step. The general idea is to differentiate $E^N_k(t)$ and follow the steps of similar proofs for the Kuramoto model. The main obstruction is that at any time $t^*$ such that $x_i^N(t^*)=x_j^N(t^*)$ for some $i,j=1,...,N$, the function $t\mapsto \nabla W(x_i^N(t)-x_j^N(t))$ is not necessarily differentiable. 
\medskip

$\diamond$ {\sc Step 2}. Proof of \eqref{E-enstrophy} for atomic measures.

To circumvent the above-mentioned problem we proceed similarly to \cite{PPS-21} for the weakly singular Kuramoto model (see also \cite{P-14} for the weakly singular Cucker-Smale system) by carefully taking care of the possible sticking and collisions between particles.
We note that since  $\nabla W$ is non-decrasing (thus, $-\nabla W$ is one sided-Lipschitz) by  Lemma \ref{szwarc} then particles stick together forming a larger cluster after sticking times. Specifically, if $j\in S_i^N(t^*)$ then $x_i^N(t^*)=x_j^N(t^*)$ and $\omega_i^N=\omega_j^N$ so that $u_i^N(t)=u_j^N(t)$ for all $t\geq t^*$. Therefore, $x_i^N(t)=x_j^N(t)$ for all $t\geq t^*$. We can then define the pure collision times by setting $T_0^N=0$ and
$$T_m^N:=\inf\{t>T_{m-1}^N:\,\exists \,i\in \{1,\ldots,N\}\  \mbox{and}\  j\notin S_i^N(T_{m-1}^N)\ \mbox{with}\  x_i^N(t)=x_j^N(t)\},$$
for any $m\in \mathbb{N}$. Note that the times $T_m^N$ account for the possibility of collision between the newly formed clusters by sticking of particles. By the H\"{o}lder regularity of $\nabla W$ we observe that $u_i^N$ appears to be only $(1-\alpha)$-H\"{o}lder continuous in time. However, by \cite[Remark 4.1]{PPS-21} we actually get the improved regularity $u_i^N\in W^{1,1}([T_{m-1}^N,\tau])$ for any $\tau\in (T_{m-1}^N,T_m^N)$ and also
$$\dot{u}_i^N(t)=\frac{1}{N}\sum_{j\notin S_i^N(T_{m-1}^N)}D^2 W(x_j^N(t)-x_i^N(t)) (u_j^N(t)-u_i^N(t)),$$
for a.e. $t\in (T_{m-1}^N,T_m^N)$. By differentiating $E_k^N$ at those times we obtain
\begin{align*}
\frac{d}{dt}E_k^N(t)&=\frac{1}{N}\sum_{i=1}^N \dot{u}_i^N(t)\cdot u_i^N(t)\\
&=\frac{1}{N^2}\sum_{i=1}^N\sum_{j\notin S_i^N(T_{m-1}^N)}u_i^N(t)^\top D^2 W(x_j^N(t)-x_i^N(t)) (u_j^N(t)-u_i^N(t))\\
&=-\frac{1}{2N^2}\sum_{i=1}^N\sum_{j\notin S_i^N(T_{m-1}^N)}[u_i^N(t)-u_j^N(t)]^\top D^2 W(x_i^N(t)-x_j^N(t)) (u_i^N(t)-u_j^N(t)),
\end{align*}
for a.e. $t\in (T_{m-1}^N,T_m^N)$, where in the last line we have symmetrized the sum by switching indices $i$ with $j$, and we have used the symmetry $D^2 W(x)=D^2 W(-x)$. By Lemma \ref{szwarc} on the right-hand side we obtain the energy dissipation inequality
\begin{align*}
    \frac{d}{dt} E_k^N(t)&\leq -D_k^N(t),\\
    D_k^N(t)&=\frac{1}{2N^2}\sum_{i=1}^N\sum_{j\notin S_i^N(T_{m-1}^N)}\vert u_i^N(t)-u_j^N(t)\vert^2\phi(\vert x_i^N(t)-x_j^N(t))
\end{align*}
for a.e. $t\in (T_{m-1}^N,T_m^N)$.
Since $E_k^N\in W^{1,1}([T_{m-1}^N,\tau])$ for any $\tau \in(T_{m-1}^N,T_k^N)$, then
$$\int_{T_{m-1}^N}^{\tau}D_k^N(t)\,dt\leq E_k^N(T_{m-1}^N)-E_k^N(\tau),$$
by integration, for any $\tau \in (T_{m-1}^N,T_m^N)$. Note that $D_k^N\geq 0$ and $E_k^N$ is continuous. Then, we can pass to the limit with $\tau$ to $T_m^N$ by monotone convergence, and sum over $m$ to get
\begin{equation}\label{E-enstrophy-discrete}
E_k^N(T)-E_k^N(0)+\int_0^T D_k^N(t)\,dt\leq 0,
\end{equation}
as long as $\{T_m^N\}_{m\in \mathbb{N}}$ spans all $[0,+\infty)$. However, it might happen that $T_m^N\rightarrow T_\infty^N<\infty$ as $m\rightarrow\infty$. By definition note that $T_\infty^N$ must then be a sticking time. Since there are exactly $N$ particles, then there must exist at most $N-1$ sticking times $T_\infty^N$. Then, we can repeat the above argument finitely many times by replacing $t=0$ by each sticking time $t=T_\infty^N$ and we cover the full lifespan of the solution by countably many intervals of the type $(T_{m-1}^N,T_m^N)$. To conclude, we claim now that $E_k^N(t)=E_k[\mu^N_t]$ and $D_k^N(t)=D_k[\mu^N_t]$ so that the above implies 
\begin{equation}\label{E-enstrophy-discrete-2}
E_k[\mu^N_T]-E_k[\mu^N_0]+\int_0^T D_k[\mu^N_t]\leq 0. 
\end{equation}
To prove such a claim, let us denote the flow of $\boldsymbol{u}[\rho_t^N]$ by $Z^N(t;x,\omega)=(X^N(t;x,\omega),\omega)$. Then, by Corollary \ref{C-Lagrangian-solution} the solution $\boldsymbol{\mu}^N$ must be a Lagrangian solution, {\it i.e.}, $\mu^N_t=Z^N(t;\cdot,\cdot)_{\#}\mu^N_0$ and we obtain
$$x_i^N(t)=X^N(t;x_{i,0},\omega_i),\quad u_i^N(t)=\boldsymbol{u}[\rho_t^N](Z^N(t;x_{i,0},\omega_i)),$$
for any $i=1,\ldots,N$ and each $t\geq 0$. Therefore,
$$E_k^N(t)=\frac{1}{2}\int_{\mathbb{R}^{2d}}\vert \boldsymbol{u}[\rho_t^N](Z^N(t;x,\omega))\vert^2\,d\mu_0^N(x,\omega)=E_k[\mu^N_t],$$
for a.e. $t\in (T_{m-1}^N,T_m^N)$. Similarly for $D_k^N$, let us define
\begin{align*}
S^N(t^*)&:=\{((x,\omega),(x',\omega'))\in\mathbb{R}^{4d}:\,Z^N(t^*;x,\omega)=Z^N(t^*;x',\omega')\},
\end{align*}
for any $t^*>0$. By construction, $S^N(t^*)$ consists of the initial configurations leading to sticking of trajectories at time $t^*>0$ and we have the relation
$$S^N(t^*)=(Z^N(t^*;\cdot,\cdot)\otimes Z^N(t^*;\cdot,\cdot))^{-1}(\Delta),$$
where $\Delta$ is the diagonal set defined in \eqref{diagonalset}. Therefore, we arrive at
\begin{multline*}
D_k^N(t)=\frac{1}{2}\int_{\mathbb{R}^{4d}\setminus S^N(t)}\vert \boldsymbol{u}[\rho_t^N](X^N(t;x,\omega),\omega)-\boldsymbol{u}[\rho_t^N](X^N(t;x',\omega'),\omega')\vert^2\\
\times\phi(\vert X^N(t;x,\omega)-X^N(t;x',\omega')\vert)\,d(\mu^N_0\otimes \mu^N_0)=D_k[\mu^N_t],
\end{multline*}
for a.e. $t\in (T_{m-1}^N,T_m^N)$, where we have used that $S_i^N(T_{m-1}^N)=S_i^N(t)$ for any $t\in (T_{m-1}^N,T_m^N)$ thanks to the absence of collisions in such an interval.

\medskip

$\diamond$ {\sc Step 3}. Proof of \eqref{E-enstrophy} for general measures.

To finish the proof, we perform a classical mean-field argument based on Proposition \ref{P-mean-field-local-in-time}. Specifically, we shall take limits as $N\rightarrow \infty$ in the energy dissipation inequality \eqref{E-enstrophy-discrete-2} in {\sc Step 2}. To that end, we must set a well-prepared sequence of empirical measures $\boldsymbol{\mu}^N$. Namely, fix any weak-measure valued solution $\boldsymbol{\mu}$ to \eqref{kurak} in the sense of Definition \ref{weakeqF} as in the statement, and consider a sequence of empirical measures $\{\boldsymbol{\mu}^N\}_{N\in\mathbb{N}}$ like in \eqref{atom} so that we have
$$\lim_{N\rightarrow\infty}\sup_{t\in [0,T]}W_2(\mu_t^N,\mu_t)=0,$$
according to Proposition \ref{P-mean-field-local-in-time}. We end our proof by noting that
\begin{align*}
E_k[\mu_0] = \lim_{N\rightarrow\infty}E_k[\mu^N_0]\quad\mbox{and}\quad
\int_0^t D_k[\mu_s]\,ds \leq \liminf_{N\rightarrow\infty}\int_0^t D_k[\mu_s^N]\,ds, \quad t\in [0,T],
\end{align*}
which will be proved in the following lemma.
\end{proof}

\medskip

\begin{lem}\label{Dlsc}
The following properties hold true:
\begin{enumerate}[label=(\roman*)]
    \item (Continuity of kinetic energy) The kinetic energy functional $E_k:\mathcal{P}_2(\mathbb{R}^{2d})\longrightarrow [0,+\infty)$ is continuous with $\mathcal{P}_2(\mathbb{R}^{2d})$ endowed with the $W_2$ metric, namely
    $$\lim_{n\rightarrow\infty}W_2(\mu^n,\mu)=0\quad \Longrightarrow \quad \lim_{n\rightarrow\infty} E_k[\mu^n]=E_k[\mu].$$
    \item (Lower semicontinuity of total dissipation) For $T>0$, define the total dissipation functional, that is, $\mathcal{D}^T: C([0,T],\mathcal{P}_2(\mathbb{R}^{2d})-W_2)\longrightarrow [0,+\infty]$ with
    $$\mathcal{D}^T[\boldsymbol{\mu}]:=\int_0^TD_k[\mu_t]\,dt,\qquad \boldsymbol{\mu}\in C([0,T],\mathcal{P}_2(\mathbb{R}^{2d})-W_2).$$
    Then, $\mathcal{D}^T$ is lower semicontinuous in $C([0,T],\mathcal{P}_2(\mathbb{R}^{2d})-W_2)$, namely
    $$\lim_{n\rightarrow\infty}\sup_{t\in [0,T]}W_2(\mu_t^n,\mu)=0\quad \Longrightarrow \quad \mathcal{D}^T[\boldsymbol{\mu}]\leq \liminf_{n\rightarrow \infty}\mathcal{D}^T[\boldsymbol{\mu}^n].$$
\end{enumerate}
\end{lem}

\noindent
Before we prove Lemma \ref{Dlsc} let us formulate yet another lemma, concerning regularity of $\boldsymbol{u}[\rho_t]$, which we will use in the sequel (including in the proof of Lemma \ref{Dlsc}).

\begin{lem}\label{L-unif-conv-u}
Let $\{\boldsymbol{\mu}^n\}_{n\in \mathbb{N}}$ and $\boldsymbol{\mu}$ belong to $C([0,T],\mathcal{P}_2(\mathbb{R}^{2d})-W_2)$ and assume that $\boldsymbol{\mu}^n\rightarrow \boldsymbol{\mu}$ in $C([0,T],\mathcal{P}_2(\mathbb{R}^{2d})-W_2)$, namely $\lim_{n\rightarrow \infty}\sup_{t\in [0,T]}W_2(\mu^n_t,\mu_t)=0$. Then, we have
$$\lim_{n\rightarrow\infty}\sup_{t\in [0,T]}| \boldsymbol{u}[\rho_t^n]-\boldsymbol{u}[\rho_t]|_\infty=0,$$
where $\rho_t^n:=\pi_{x\#}\mu_t^n$ and $\rho_t:=\pi_{x\#}\mu_t$ for $t\in [0,T]$.
\end{lem}

\begin{proof}
Observe that we can write
$$\boldsymbol{u}[\rho_t^n](x,\omega)-\boldsymbol{u}[\rho_t](x,\omega)=I_{n,\varepsilon}^1(t,x,\omega)+I_{n,\varepsilon}^2(t,x,\omega),$$
for any $\varepsilon\in (0,1)$, where we have
\begin{align*}
I_{n,\varepsilon}^1(t,x,\omega)&:=-\int_{\mathbb{R}^{2d}}\xi_\varepsilon(|x-x'|)\,\nabla W(x-x')\,(d\mu_t^n(x',\omega')-d\mu_t(x',\omega')),\\
I_{n,\varepsilon}^2(t,x,\omega)&:=-\int_{\mathbb{R}^{2d}}(1-\xi_\varepsilon(|x-x'|))\,\nabla W(x-x')\,(d\mu_t^n(x',\omega')-d\mu_t(x',\omega')),
\end{align*}
for a smooth cut-off function $\xi_\varepsilon$ verifying $0\leq \xi_\varepsilon\leq 1$ with $\xi_\varepsilon\equiv 0$ in $[0,\varepsilon]$ and $\xi_\varepsilon\equiv 1$ in $[2\varepsilon,+\infty)$. On the one hand, we have
$$|(1-\xi_\varepsilon(|x-x'|))\,\nabla W(x-x')|\lesssim |x-x'|^{1-\alpha}\mathds{1}_{|x-x'|\leq 2\varepsilon}\lesssim \varepsilon^{1-\alpha},$$
which for the second term yields the uniform estimate 
\begin{equation}\label{E-unif-conv-u-1}
| I_{n,\varepsilon}^2|_\infty\lesssim \varepsilon^{1-\alpha}.
\end{equation}
On the other hand, let us define the Lipschitz functions $\varphi_{x,\varepsilon}(x'):=\nabla W(x-x')\xi_\varepsilon(|x-x'|)$ in the integrand of $I_{n,\varepsilon}^1$. Since $\varphi_{x,\varepsilon}(x')=\varphi_{0,\varepsilon}(x'-x)$, then its Lipscschitz seminorm is independent of $x$ and we have $[\varphi_{x,\varepsilon}]_{\rm Lip}=[\varphi_{0,\varepsilon}]_{\rm Lip}$. Using the Kantorovitch-Rubinstein duality theorem and interpolating $W_2$ by $W_1$ Wasserstein distance we have
\begin{equation}\label{E-unif-conv-u-2}
| I_{n,\varepsilon}^1|_\infty\leq [\varphi_{0,\varepsilon}]_{\rm Lip} \sup_{t\in [0,T]} W_2(\mu_t^n,\mu_t).
\end{equation}
Hence, putting \eqref{E-unif-conv-u-1}-\eqref{E-unif-conv-u-2} together, taking limits as $n\rightarrow \infty$ and by the convergence $\boldsymbol{\mu}^n\rightarrow \boldsymbol{\mu}$ in $C([0,T],\mathcal{P}_2(\mathbb{R}^{2d})-W_2)$ we obtain
$$\limsup_{n\rightarrow \infty}\sup_{t\in [0,T]}| \boldsymbol{u}[\rho_t^n]-\boldsymbol{u}[\rho_t]|_\infty\lesssim \varepsilon^{1-\alpha}.$$
Since the choice of $\varepsilon\in (0,1)$ is arbitrary, this finishes the proof.
\end{proof}

\medskip

\begin{proof}[Proof of Lemma \ref{Dlsc}]
~

\smallskip

$\diamond$ {\sc Step 1}. Proof of $(i)$.\\
Let $\{\mu^n\}_{n\in \mathbb{N}}$ and $\mu$ belong to $\mathcal{P}_2(\mathbb{R}^{2d})$ with $\mu_n\rightarrow \mu$ in the $W_2$ metric. Note that
$$E_k[\mu^n]-E_k[\mu]=I_n^1+I_n^2,$$
for every $n\in \mathbb{N}$, where each term reads
\begin{align*}
I_n^1&:=\frac{1}{2}\int_{\mathbb{R}^{2d}}(\boldsymbol{u}[\rho^n](x,\omega)-\boldsymbol{u}[\rho](x,\omega))\cdot (\boldsymbol{u}[\rho^n](x,\omega)+\boldsymbol{u}[\rho](x,\omega)) \,d\mu^n(x,\omega),\\
I_n^2&:=\frac{1}{2}\int_{\mathbb{R}^{2d}}|\boldsymbol{u}[\rho](x,\omega)|^2(d\mu^n-d\mu).
\end{align*}
On the one hand, by the Cauchy--Schwarz inequality we have
\begin{align*}
\vert I_n^1\vert &\leq \frac{1}{\sqrt{2}}| \boldsymbol{u}[\rho^n]-\boldsymbol{u}[\rho]|_\infty\left(\int_{\mathbb{R}^{2d}}|\boldsymbol{u}[\rho_n](x,\omega)|^2+|\boldsymbol{u}[\rho](x,\omega)|^2\,d\mu^n(x,\omega)\right)^{1/2}\\
&\lesssim | \boldsymbol{u}[\rho^n]-\boldsymbol{u}[\rho]|_\infty \left(1+\int_{\mathbb{R}^{2d}}|x|^2+|\omega|^2\,d\mu^n(x,\omega)\right),
\end{align*}
where we have used that $\boldsymbol{u}[\rho^n]$ and $\boldsymbol{u}[\rho]$ are continuous with subquadratic growth
\begin{align*}
    |\boldsymbol{u}[\rho^n](x,\omega)|^2+|\boldsymbol{u}[\rho](x,\omega)|^2 \lesssim 1+|x|^2 + |\omega|^2,
\end{align*}
uniformly in $n\in \mathbb{N}$. Therefore, we obtain $\lim_{n\rightarrow\infty} I_n^1=0$ by the uniform convergence of $\boldsymbol{u}[\rho_n]$ toward $\boldsymbol{u}[\rho]$ in Lemma \ref{L-unif-conv-u}, and by the fact the second order moments of $\mu^n$ are uniformly bounded ({\it cf.} Proposition \ref{W2cont}). On the other hand, recall again that $|\boldsymbol{u}[\rho]|^2$ is a continuous function with quadratic growth. Since $\mu^n\rightarrow \mu$ in the $W_2$ metric, the usual characterization of the convergence in $W_2$ in terms of test functions shows that $\lim_{n\rightarrow\infty} I_n^2=0$. Altogether shows that $E_k[\mu^n]\rightarrow E[\mu]$.

\medskip

$\diamond$ {\sc Step 2}. Proof of $(ii)$.\\
For any $m>0$ we define the truncated functional $\mathcal{D}_m^T:C([0,T],\mathcal{P}_2(\mathbb{R}^{2d})-W_2)\longrightarrow [0,+\infty)$ as
\begin{align*}
\mathcal{D}_m^T[\boldsymbol{\mu}]:= \int_0^T\int_{\RR^{4d}}\min\Big\{\phi(|x-x'|),m\Big\}\min\Big\{|\boldsymbol{u}[\rho_t](x,\omega)-\boldsymbol{u}[\rho_t](x',\omega')|^2,m\Big\}\,d(\mu_t\otimes\mu_t)\,dt,
\end{align*}
for any $\boldsymbol{\mu}\in C([0,T],\mathcal{P}_2(\mathbb{R}^{2d})-W_2)$. Note that, by the monotone convergence theorem we have
$$\mathcal{D}^T[\boldsymbol{\mu}]=\sup_{m>0}\mathcal{D}_m^T[\boldsymbol{\mu}].$$
Therefore, the lower semicontinuity will follow as along as we prove that the truncated functionals $\mathcal{D}_m^T$ are continuous in $C([0,T],\mathcal{P}_2(\mathbb{R}^{2d})-W_2)$ for all $m>0$. To do so, take any sequence $\{\boldsymbol{\mu}^n\}_{n\in \mathbb{N}}$ convergent to $\boldsymbol{\mu}$ in $C([0,T],\mathcal{P}_2(\mathbb{R}^{2d})-W_2)$ and notice that
$$\mathcal{D}_m^T[\boldsymbol{\mu}^n]-\mathcal{D}_m^T[\boldsymbol{\mu}]=II_n^1+II_n^2,$$
where each term reads
\begin{align*}
    II_n^1&:=\int_0^T\int_{\mathbb{R}^{4d}}\min\Big\{\phi(|x-x'|),m\Big\}\Big[\min\Big\{|\boldsymbol{u}[\rho_t^n](x,\omega)-\boldsymbol{u}[\rho_t^n](x',\omega')|^2,m\Big\}\\
    &\qquad -\min\Big\{|\boldsymbol{u}[\rho_t](x,\omega)-\boldsymbol{u}[\rho_t](x',\omega')|^2,m\Big\}\Big]\,d(\mu_t^n\otimes \mu_t^n)\,dt\\
    II_n^2&:=\int_0^T\int_{\mathbb{R}^{4d}}\min\Big\{\phi(|x-x'|),m\Big\}\min\Big\{|\boldsymbol{u}[\rho_t](x,\omega)-\boldsymbol{u}[\rho_t](x',\omega')|^2,m\Big\}\\
    &\qquad \times \,(d(\mu^n_t\otimes \mu^n_t)-d(\mu_t\otimes \mu_t))\,dt.
\end{align*}
On the one hand, since the function $\xi\in \mathbb{R}^d\mapsto \min\{|\xi|^2,m\}$ is Lipschitz-continuous with constant $2\sqrt{m}$, we have for the first term
$$|II_n^1|\leq 4\,T\,m^{3/2}\sup_{t\in [0,T]}| \boldsymbol{u}[\rho_t^n]-\boldsymbol{u}[\rho_t]|_\infty,$$
which converges to zero as $n\rightarrow \infty$ by Lemma \ref{L-unif-conv-u}. On the other hand, for every $t\in [0,T]$ the integrand appearing in the second term $II_n^2$ is a bounded continuous function in $\mathbb{R}^{4d}$. Since $\mu_t^n\otimes \mu_t^n\rightarrow \mu_t\otimes \mu_t$ narrowly, we have
$$\int_{\mathbb{R}^{4d}}\min\Big\{\phi(|x-x'|),m\Big\}\min\Big\{|\boldsymbol{u}[\rho_t](x,\omega)-\boldsymbol{u}[\rho_t](x',\omega')|^2,m\Big\}\,(d(\mu^n_t\otimes \mu^n_t)-d(\mu_t\otimes \mu_t))\rightarrow 0,$$
for every $t\in [0,T]$. Moreover, as a function of $t$, the latter is upper-bounded by $2m^2$, which is integrable in $[0,T]$. Thus, by the dominated convergence theorem, we also have $II_n^2\rightarrow 0$ as $n\rightarrow\infty$, and this ends the proof.
\end{proof}

\section{Proof of the main theorems}\label{sec:mainproof}
The rest of the paper is dedicated to the proof of Theorems \ref{main1}, \ref{main2} and \ref{main3}. In Sections \ref{sec:KtoCS} and \ref{sec:CStoK} we prove Theorem \ref{main1}. Then in Section \ref{sec:Cfinalp} we briefly apply Theorem \ref{main1} together with earlier results from \cite{PP-22-1-arxiv} to prove Theorems \ref{main2} and \ref{main3}. Let us begin by making a remark regarding testing of our main equations \eqref{sysk} and \eqref{kurak} by time-independent test functions.

\begin{rem}[Testing by time-independent functions]\label{test}
For any solution $\boldsymbol{f}$ to \eqref{sysk} in the sense of Definition \ref{DCS}, we can follow the lines of the proof of Proposition \ref{W2contf}, particularly {\sc Step 3} related to claim \eqref{E-second-moments-x} with a test function $\varphi(x,v)=g(x)\in C^1_b(\RR^d)$, and we arrive at
\begin{align*}
    \frac{1}{t-s}\left(\int_{\RR^{2d}} g(x)\ df_t - \int_{\RR^{2d}}g(x)\, df_s\right) = \frac{1}{t-s}\int_s^t \int_{\RR^{2d}} \nabla_x g(x)\cdot v\, df_r\, dr,
\end{align*}
for all $s\leq t$ in $[0,T]$. Taking limits $t\to s$ we obtain
\begin{align}\label{test2}
    \frac{d}{dt} \int_{\RR^{2d}} g(x)\, df_t = \int_{\RR^{2d}} \nabla_x g(x)\cdot v\, df_t,
\end{align}
for every $t\in [0,T]$ of continuity of the right-hand side. By Proposition \ref{W2contf} these are all $t\in[0,T]$, and therefore $t\mapsto \int_{\RR^{2d}} g\,df_t$ belongs to $C^1([0,T])$. A similar argument can be performed with any solution $\boldsymbol{\mu}$ to \eqref{kurak} in the sense of Definition \ref{D1st}, based on Proposition \ref{W2cont}. Thus for any $g\in C^1(\RR^d)$, such that $\nabla_x g$ has sublinear growth, we have that $t\mapsto \int_{\RR^{2d}} g \,d\mu_t$ is $C^1([0,T])$ and
\begin{align}\label{test1}
\frac{d}{dt}\int_{\RR^{2d}}g\,d\mu_t = -\int_{\RR^{2d}}(\omega-\nabla W*\rho_t)\cdot\nabla_x g \,d\mu_t,
\end{align}
for any $t\in [0,T]$. Taking $g(x-\cdot)$ as the test function in \eqref{test2} and \eqref{test1}, and recalling the notation notation $V*[W]$ in \eqref{E-convolution-vector-vector}, leads to the identities
\begin{align*}
&\partial_t(g* f_t) = -(\nabla_x g)*[vf_t]:= -\int_{\RR^{2d}}\nabla_x g(x-x')\cdot v'\,df_t',\\
&\partial_t (g*\mu_t) = -(\nabla_x g)* [(\omega-\nabla W*\rho_t)\mu]:= -\int_{\RR^{2d}}\nabla_x g(x-x')\cdot(\omega'-\nabla W*\rho_t(x'))\,d\mu_t',
\end{align*}
which we use in the proof of the main result Theorem \ref{main1}.
\end{rem}

\subsection{From first order Kuramoto system to the second order alignment dynamics}\label{sec:KtoCS}

For the readers convenience let us restate the first part of Theorem \ref{main1} in a more compact way.
\begin{pro}\label{PCSexi}
Suppose that $\alpha\in(0,1)$ and $T>0$ and let $\boldsymbol{\mu}$ be a weak measure-valued solution to \eqref{kurak} in the sense of Definition \ref{D1st}. Assume further that
\begin{align}\label{nudiag}
    \int_{\RR^{2d}}|\omega-\omega'|^\frac{2-3\alpha}{1-\alpha}\, d(\nu\otimes\nu)<\infty.
\end{align}
Then, $f_t=(\cT^{1\to 2}[\rho_t])_{\#} \mu_t$, with change of variables $\mathcal{T}^{1\to 2}[\rho_t]$ given in Definition \ref{D-change-variables}, is a weak solution of \eqref{sysk} in the sense of Definition \ref{DCS}.
\end{pro}

\begin{proof}  First, observe that by  Proposition \ref{W2cont} and Lemma \ref{enstrophy}, and by the change of variables formula between $\boldsymbol{\mu}$ and $\boldsymbol{f}$, we obtain that the moment inequality \eqref{weakeqfb} and the energy dissipation inequality \eqref{weakeqfa} are readily satisfied. Next, we show that $\boldsymbol{f}\in C([0,T];\mathcal{P}(\mathbb{R}^{2d})-\mbox{narrow})$ using the characterization of the narrow convergence by testing against bounded-Lipschitz test function by virtue of Portmanteau's theorem. Indeed, given $\{t_n\}_{n\in \mathbb{N}}$ and $t_0$ in $[0,T]$ with $t_n\rightarrow t_0$, and a bounded-Lipschitz test function $g$, we have
\begin{align*}
    &\int_{\RR^{2d}} g\, df_{t_n} - \int_{\RR^{2d}} g\, df_{t_0} = \int_{\RR^{2d}} g \circ \cT^{1\to 2}[\rho_{t_n}]\, d\mu_{t_n} - \int_{\RR^{2d}} g\circ \cT^{1\to 2}[\rho_{t_0}]\, d\mu_{t_0}\\
    &\quad= \left(\int_{\RR^{2d}} g(x, \boldsymbol{u}[\rho_{t_n}](x,\omega))\, d\mu_{t_n} - \int_{\RR^{2d}} g(x, \boldsymbol{u}[\rho_{t_0}](x,\omega))\, d\mu_{t_n}\right)\\
    &\qquad\qquad+ \left(\int_{\RR^{2d}}g(x, \boldsymbol{u}[\rho_{t_0}](x,\omega))\, d\mu_t - \int_{\RR^{2d}} g(x, \boldsymbol{u}[\rho_{t_0}](x,\omega))\, d\mu_{t_0}\right)=: I_n + II_n,
\end{align*}
by the pushforward change of variables formula. On the one hand, Proposition \ref{W2cont} ensures that $\boldsymbol{\mu}\in C([0,T],\mathcal{P}_2(\mathbb{R}^{2d})-W_2)$, which implies that $\mu_{t_n}\to\mu_{t_0}$ in $W_2$ and therefore
$$|I_n|\leq [g]_{\rm Lip}| \boldsymbol{u}[\rho_{t_n}]-\boldsymbol{u}[\rho_{t_0}]|_\infty\rightarrow 0,$$
by Lemma \ref{L-unif-conv-u}. On the other hand, for $II_n$ we again use the convergence $\mu_{t_n}\rightarrow \mu_{t_0}$ in $W_2$, and more specifically, its characterization by convergence of integrals tested with continuous functions with subquadratic growth, see \cite[Theorem 6.9]{V-09}. Indeed, note that the continuous test function $g\circ \mathcal{T}^{1\to 2}[\rho_{t_0}]$ in the integrand has subquadratic growth since we have
\begin{align*}
|g\circ \mathcal{T}^{1\to 2}(x,\omega)|&\leq |g(0,0)|+|g\circ \mathcal{T}^{1\to 2}(x,\omega)-g(0,0)|\\
&\leq |g(0,0)|+[g]_{\rm Lip}(|x|+|\boldsymbol{u}[\rho_{t_0}](x,\omega)|)\\
&\lesssim 1+|x|^2+|\omega|^2,
\end{align*}
by the Lipschitz property of $g$ and the growth of the velocity field in Corollary \ref{C-sided-Lipschitz}(iii). Hence, we obtain that $II_n\to 0$. The remainder of the proof, {\it i.e.}, that $\boldsymbol{f}$ satisfies the weak formulation \eqref{weakeqf} in Definition \ref{DCS}, follows by steps.

\medskip

$\diamond$ {\sc Step 1}. Preparing the test function.\\
The ultimate goal of this proposition is to show that if $\boldsymbol{\mu}$ satisfies \eqref{weakeqF} then $\boldsymbol{f}$ satisfies \eqref{weakeqf}. To this end, let $\psi$ be any test function from Definition \ref{DCS}. By Proposition \ref{W2cont} we could consider an even more general $\psi\in C^1_b([0,T]\times\RR^{2d})$, compactly supported in $[0,T)$, with uniformly Lipschitz-continuous $\nabla_v\psi$. Our goal is to pull-back $\psi$ along the mapping $\cT^{2\to 1}$ and use it as a test function in \eqref{kurak}. However, due to the insufficient regularity of $\cT^{2\to 1}$, we need to mollify first. Fix $\varepsilon> 0$ and let us define the test function
\begin{align*}
\eta_t(x,\omega) &:=  \psi_t(x, \boldsymbol{u}_\varepsilon[\rho_t](x,\omega)),\\ 
\boldsymbol{u}_\varepsilon[\rho_t](x,\omega) &:= \omega- \nabla W_\varepsilon*\rho_t(x),
\end{align*}
where we have set
\begin{align}\label{epep}
\begin{aligned}
W_\varepsilon(x) &:=\frac{1}{2-\alpha}(|x|+\varepsilon)^{1-\alpha}\left(|x|-\frac{\varepsilon}{1-\alpha}\right),\\
\nabla W_\varepsilon (x) &= \frac{1}{1-\alpha} x (|x|+\varepsilon)^{-\alpha},\\
D^2 W_\varepsilon(x) &= \frac{1}{1-\alpha}(|x|+\varepsilon)^{-\alpha} \left(I - \alpha \frac{x}{|x|}\otimes\frac{x}{|x|+\varepsilon}\right),
\end{aligned}
\end{align}
which can be viewed as mollified versions of $W$, $\nabla W$ and $D^2 W$ respectively.

Our goal is to show that $\eta$ is an admissible test function in Definition \ref{D1st}. First, let us make a couple of observations regarding regularity of $\boldsymbol{u}[\rho_t]$ and $\boldsymbol{u}_\varepsilon[\rho_t]$ in the following lemma, which in particular improves the simpler space continuity in Corollary \ref{C-sided-Lipschitz} to space-time continuity.

\begin{lem}[Regularity of mollifications]\label{partialtu}
Let $\boldsymbol{\mu}$ be any solution to \eqref{kurak} in the sense of Definition \ref{D1st}. Then, the velocity fields $\boldsymbol{u}[\rho]$ and $\boldsymbol{u}_\varepsilon[\rho]$ belong to the class $C([0,T]\times \RR^{2d},\mathbb{R}^d)$ and they have sublinear growth for any $\varepsilon>0$. Moreover, we have 
\begin{align*}
&[D^2W_\varepsilon*(\boldsymbol{u}[\rho]\mu)]\in C_b([0,T]\times\RR^{2d},\RR^d),\\ &D^2W_\varepsilon*\mu\in C_b([0,T]\times\RR^{2d},\RR^{d\times d}),
\end{align*}
and the following identities hold true
\begin{align*}
\partial_t\boldsymbol{u}_\varepsilon[\rho_t] &= [D^2W_\varepsilon*(\boldsymbol{u}[\rho_t]\mu_t)],\\ \nabla_x\boldsymbol{u}_\varepsilon[\rho_t] &= D^2W_\varepsilon*\mu_t,
\end{align*}
for every $(t,x,\omega)\in [0,T]\times \mathbb{R}^{2d}$. Here, we recall that
$$[D^2 W_\varepsilon*(\boldsymbol{u}[\rho_t]\mu_t)](x)=\int_{\mathbb{R}^{2d}} D^2 W_\varepsilon(x-x')\boldsymbol{u}[\rho_t](x',\omega')\,d\mu_t(x',\omega'),$$
in accordance with the notation $[M*V]$ in \eqref{E-convolution-matrix-vector}.
\end{lem}

\begin{proof}
First, we prove that $\boldsymbol{u}[\rho]\in C([0,T]\times \RR^{2d},\mathbb{R}^d)$. Consider any sequence $\{(t_n,x_n,\omega_n)\}_{n\in \mathbb{N}}$ and $(t,x,\omega)$ in $[0,T]\times \mathbb{R}^{2d}$ such that $(t_n,x_n,\omega_n)\to(t,x,v)$. Then, we have
\begin{multline*}
    \boldsymbol{u}[\rho_{t_n}](x_n,\omega_n) - \boldsymbol{u}[\rho_{t}](x,\omega) \\
    = \big(\boldsymbol{u}[\rho_{t_n}](x_n,\omega_n) - \boldsymbol{u}[\rho_{t}](x_n,\omega_n)\big) + \big(\boldsymbol{u}[\rho_{t}](x_n,\omega_n) - \boldsymbol{u}[\rho_{t}](x,\omega)\big)=: I_n + II_n.
\end{multline*}
Note that Proposition \ref{W2cont} implies that $\mu_{t_n}\to\mu_t$ in $W_2$, which together with Lemma \ref{L-unif-conv-u} ensures that $I_n\to 0$. In addition, the convergence $II_n\to 0$ follows from Corollary \ref{C-sided-Lipschitz}(i). Thus, we have $\boldsymbol{u}[\rho]\in C([0,T]\times \RR^{2d},\mathbb{R}^d)$ and by analogous arguments $\boldsymbol{u}_\varepsilon[\rho]\in C([0,T]\times \RR^{2d},\mathbb{R}^d)$. To finish the proof first note that the mollified weight verifies that $\nabla W_\varepsilon\in C^1(\mathbb{R}^d,\mathbb{R}^d)$, it has subquadratic growth, and $D^2W_\varepsilon\in C_b(\RR^{d},\mathbb{R}^{d\times d})$. Thus, recalling Remark \ref{test}, we have
\begin{align*}
    \partial_t\boldsymbol{u}_\varepsilon[\rho_t] = -\partial_t(\nabla W_\varepsilon*\rho_t) = [D^2W_\varepsilon*(\boldsymbol{u}[\rho_t]\mu_t)].
\end{align*}
Finally, by Young's inequality for convolutions, since $D^2W_\varepsilon\in C_b(\RR^{d},\RR^{d\times d})$, we then have $[D^2W_\varepsilon*(\boldsymbol{u}[\rho]\mu)]\in C_b([0,T]\times\RR^{2d},\RR^d)$. Since the proof of the remaining properties involving $\nabla_x\boldsymbol{u}_\varepsilon[\rho_t]$ is analogous, we omit it here.
\end{proof}

\smallskip

By the above lemma and the chain rule
\begin{align}\label{dteta}
\begin{aligned}
\partial_t\eta_t(x,\omega) &= \partial_t\psi_t(x,\boldsymbol{u}_\varepsilon[\rho_t](x,\omega)) + \nabla_v\psi_t(x,\boldsymbol{u}_\varepsilon[\rho_t](x,\omega))\cdot [D^2W_\varepsilon*(\boldsymbol{u}[\rho_t]\mu_t)](x),\\
\nabla_x\eta_t(x,\omega) &= \nabla_x\psi_t(x,\boldsymbol{u}_\varepsilon[\rho_t](x,\omega)) - \nabla_v\psi_t(x,\boldsymbol{u}_\varepsilon[\rho_t](x,\omega))^\top(D^2W_\varepsilon * \mu_t(x)),
\end{aligned}
\end{align}
and both above derivatives belong to $C_b([0,T]\times\RR^{2d})$, which ensures that $\eta\in C^1_b([0,T]\times\RR^{2d})$ is an admissible test function in Definition \ref{D1st} (it inherits compact support in $[0,T)$ from $\psi$).

\medskip

$\diamond$ {\sc Step 2}. Transforming the equations.\\
In {\sc Step 1} we established that the $\varepsilon$-dependent function $\eta$ is an admissible test function for the weak formulation \eqref{weakeqF}. Therefore, we have
\begin{align}\label{step31}
\int_0^T\int_{\RR^{2d}} \Big(\partial_t\eta + \boldsymbol{u}[\rho_t]\cdot\nabla_x\eta\Big)\, d \mu_t\, dt = - \int_{\RR^{2d}} \eta_0 \, d \mu_0,
\end{align}
or equivalently, by \eqref{dteta},
\begin{multline}\label{diagl2}
{\mathcal L}_1 + {\mathcal L}_2 := \int_0^T\int_{\RR^{2d}}\Big (\partial_t\psi_t(x,\boldsymbol{u}_\varepsilon[\rho_t](x,\omega)) + \boldsymbol{u}[\rho_t](x,\omega)\cdot\nabla_x\psi_t(x,\boldsymbol{u}_\varepsilon[\rho_t](x,\omega) )\Big)\, d\mu_t\, dt \\
+ \int_0^T\int_{\RR^{2d}}\nabla_v\psi_t(x, \boldsymbol{u}_\varepsilon[\rho_t](x,\omega))\cdot\Big(\left[D^2W_\varepsilon* (\boldsymbol{u}[\rho_t]\mu_t)\right](x) -\left(D^2W_\varepsilon *\rho_t(x)\right)\boldsymbol{u}[\rho_t](x) \Big)\, d\mu_t\, dt \\
= - \int_{\RR^{2d}} \psi_0(x, \boldsymbol{u}[\rho_0](x,\omega))\, d \mu_0 =: {\mathcal R}.
\end{multline}
We diagonalize ${\mathcal L}_2$, so that it better corresponds to the nonlinear term in Definition \ref{DCS}. For the sake of notational simplicity we define the measure
\begin{align}\label{notat}
\lambda : = \mu_t\otimes\mu_t\otimes {\mathcal L}^1_{\mathlarger{\llcorner}[0,T]}(t),
\end{align}
where ${\mathcal L}^1_{\mathlarger{\llcorner}[0,T]}$ is the 1D Lebesgue measure restricted to $[0,T]$.
Since $D^2W_\varepsilon(x) = D^2W_\varepsilon(-x)$, then we can perform the change of variables $(x,\omega)\leftrightarrow (x',\omega')$ to obtain the symmetric form
\begin{align}\label{E-L2-symmetrized}
{\mathcal L}_2 &= \int_0^T\int_{\RR^{4d}} \nabla_v\psi(x,\boldsymbol{u}_\varepsilon[\rho_t])^\top D^2W_\varepsilon(x-x')\left(\boldsymbol{u}[\rho_t]'-\boldsymbol{u}[\rho_t]\right)\, d\lambda\nonumber\\
&=\frac{1}{2}\int_0^T\int_{\RR^{4d}} \left(\nabla_v\psi(x,\boldsymbol{u}_\varepsilon[\rho_t]) - \nabla_v\psi(x,\boldsymbol{u}_\varepsilon[\rho_t])'\right)^\top  D^2W_\varepsilon(x-x')\left(\boldsymbol{u}[\rho_t]'-\boldsymbol{u}[\rho_t]\right)\, d\lambda.
\end{align}

\medskip

$\diamond$ {\sc Step 3}. Passing with $\varepsilon\to 0$; two useful lemmas.\\
We continue by proving the following lemmas.

\begin{lem}\label{phiuniform}
Consider any $\boldsymbol{\mu}\in C([0,T],\mathcal{P}_2(\mathbb{R}^{2d})-W_2)$. Then, we have 
$$\lim_{\varepsilon\rightarrow 0}\sup_{t\in [0,T]}| \boldsymbol{u}_\varepsilon[\rho_t]-\boldsymbol{u}[\rho_t]|_\infty=0.$$
\end{lem}

\begin{proof}
It suffices to see that
\begin{align*}
|\nabla W_\varepsilon *\rho_t(x) - \nabla W *\rho_t(x) |&= \left|\int_{\RR^{2d}} (x-x')\Big[(|x-x'|+\varepsilon)^{-\alpha} - |x-x'|^{-\alpha}\Big] \,d\mu_t(x',\omega')\right|\\
&\leq \int_{\RR^{2d}} |x-x'|\Big[|x-x'|^{-\alpha}- (|x-x'|+\varepsilon)^{-\alpha} \Big] \,d\mu_t(x',\omega')\\
& = \int_{\RR^{2d}}\Big[|x-x'|^{1-\alpha} - (|x-x'|+\varepsilon)^{1-\alpha}\Big] \,d\mu_t(x',\omega')\\
&+ \varepsilon\int_{\RR^{2d}}(|x-x'|+\varepsilon)^{-\alpha}\,d\mu_t(x',\omega')\leq 2\varepsilon^{1-\alpha}.
\end{align*}
\end{proof}

\begin{lem}\label{gepg}
Let $\boldsymbol{\mu}$ be any solution to \eqref{kurak} in the sense of Definition \ref{D1st}, and consider
\begin{align*}
g_\varepsilon(t,x,x',\omega,\omega')&:=\big|\boldsymbol{u}_\varepsilon[\rho_t]- \boldsymbol{u}_\varepsilon[\rho_t]'\big|\ \big|\boldsymbol{u}[\rho_t]-\boldsymbol{u}[\rho_t]'\big|\ \phi_\varepsilon(|x-x'|),\\
g(t,x,x',\omega,\omega') &:=\big|\boldsymbol{u}[\rho_t]-\boldsymbol{u}[\rho_t]'\big|^2\ \phi(|x-x'|),
\end{align*}
for any $\varepsilon>0$, assumed to be equal to $0$ on the diagonal set $\Delta$ defined in \eqref{diagonalset}, where we set the truncation $\phi_\varepsilon(r):= (r+\varepsilon)^{-\alpha}$ for $r\geq 0$. Then, the following inequality is satisfied 
\begin{align}\label{gepg01}
    g_\varepsilon\lesssim
    \left\{
    \begin{array}{ll}
        \vert x-x'\vert^{2-3\alpha}+g,  & \mbox{ if } \alpha\in(0,\frac{2}{3}],  \\
       \vert \omega-\omega'\vert^\frac{2-3\alpha}{1-\alpha}+g,  & \mbox{ if } \alpha\in(\frac{2}{3},1).
    \end{array}
    \right.
\end{align}
In both cases, the right-hand side above is integrable with respect to $\lambda  = \mu_t\otimes\mu_t\otimes {\mathcal L}^1_{\mathlarger{\llcorner}[0,T]}(t)$.
\end{lem}

\begin{proof}
A straightforward computation shows that, the function $x\mapsto x(|x|+\varepsilon)^{-\alpha} = (1-\alpha)W_\varepsilon(x)$ is H\" older-continuous for every $\varepsilon>0$ with an explicit $\varepsilon$-independent constant $H\leq 5$ and exponent $1-\alpha$. With this information in mind, we define
\begin{align*}
A := \Big\{(x,x',\omega,\omega')\in \mathbb{R}^{4d}:\,|\omega-\omega'|\leq \frac{3H}{1-\alpha}|x-x'|^{1-\alpha}\Big\}.
\end{align*}
On the one hand, we note that on $A$ we have
\begin{align*}
\big|\boldsymbol{u}[\rho_t]-\boldsymbol{u}[\rho_t]'\big|\leq |\omega-\omega'| + \frac{H}{1-\alpha}|x-x'|^{1-\alpha}\leq \frac{4H}{1-\alpha}|x-x'|^{1-\alpha},\\
\big|\boldsymbol{u}_\varepsilon[\rho_t]-\boldsymbol{u}_\varepsilon[\rho_t]'\big|\leq |\omega-\omega'| + \frac{H}{1-\alpha}|x-x'|^{1-\alpha}\leq \frac{4H}{1-\alpha} |x-x'|^{1-\alpha},
\end{align*}
and therefore
\begin{equation}\label{E-gepsilon-inside}
g_\varepsilon(t,x,x',\omega,\omega')\lesssim |x-x'|^{2-3\alpha},
\end{equation}
for every $t\in [0,T]$ and $(x,x',\omega,\omega')\in A$. On the other hand, outside of $A$ we have
\begin{align*}
\Big|\big|\boldsymbol{u}_\varepsilon[\rho_t]-&\boldsymbol{u}_\varepsilon[\rho_t]'\big|-\big|\boldsymbol{u}[\rho_t]-\boldsymbol{u}[\rho_t]'\big|\Big| \leq \Big|\big(\boldsymbol{u}_\varepsilon[\rho_t]-\boldsymbol{u}[\rho_t]\big)-\big(\boldsymbol{u}_\varepsilon[\rho_t]'-\boldsymbol{u}[\rho_t]'\big)\Big|\\
&= 
\Big|(\nabla W_\varepsilon*\mu_t(x)-\nabla W*\mu_t(x))-(\nabla W_\varepsilon*\mu_t(x')-\nabla W*\mu_t(x'))\Big|\\
&=\Big|(\nabla W_\varepsilon*\mu_t(x)-\nabla W_\varepsilon*\mu_t(x'))-(\nabla W*\mu_t(x)-\nabla W*\mu_t(x'))\Big|\\
&\leq
\frac{2H}{1-\alpha} |x-x'|^{1-\alpha}
\leq |\omega-\omega'|- \frac{H}{1-\alpha}|x-x'|^{1-\alpha}\\
&\leq \big|\boldsymbol{u}[\rho_t]-\boldsymbol{u}[\rho_t]'\big|,
\end{align*}
which implies
\begin{align*}
\big|\boldsymbol{u}_\varepsilon[\rho_t]-\boldsymbol{u}_\varepsilon[\rho_t]'\big|\leq 2\big|\boldsymbol{u}[\rho_t]-\boldsymbol{u}[\rho_t]'\big|,
\end{align*}
and thus also
\begin{equation}\label{E-gepsilon-outside}
g_\varepsilon(t,x,x',\omega,\omega')\leq 2g(t,x,x',\omega,\omega'),
\end{equation}
for every $t\in [0,T]$ and $(x,x',\omega,\omega')\in \mathbb{R}^{4d}\setminus A$.

Putting \eqref{E-gepsilon-inside} and \eqref{E-gepsilon-outside} together in the case $\alpha\in (0,\frac{2}{3}]$ we arrive to the first way $\eqref{gepg01}_1$ of dominating $g_\varepsilon$. Note that in that case the right hand side belongs to $L^1_\lambda([0,T]\times \mathbb{R}^{4d})$ because so does $g$ by Lemma \ref{enstrophy}, and $|x-x'|^{2-3\alpha}$ by interpolation by the (integrable) zeroth and second order moments. When, $\alpha\in (\frac{2}{3},1)$ quantities appearing in \eqref{E-gepsilon-inside} can be further controlled by
$$g_\varepsilon(t,x,x',\omega')\lesssim |x-x'|^{2-3\alpha}\lesssim |\omega-\omega'|^\frac{2-3\alpha}{1-\alpha},$$
for every $t\in [0,T]$ and $(x,x',\omega,\omega')\in A$, thus yielding the second way $\eqref{gepg01}_2$ of dominating  $g_\varepsilon$. The right hand side now belongs to $L^1_\lambda([0,T]\times \mathbb{R}^{4d})$ by virtue of the assumption \eqref{nudiag0}. 
\end{proof}

\medskip

$\diamond$ {\sc Step 4}. Passing with $\varepsilon\to 0$; the case of ${\mathcal L}_1$ and ${\mathcal R}$.\\
We proceed with the limit as $\varepsilon\to 0$.
By Lemma \ref{phiuniform} and the regularity of $\psi$ we have 
$$\partial_t\psi_t(x,\boldsymbol{u}_\varepsilon[\rho_t])\to\partial_t\psi_t(x,\boldsymbol{u}[\rho_t])\quad \mbox{and}\quad \nabla_x\psi_t(x,\boldsymbol{u}_\varepsilon[\rho_t])\to\nabla_x\psi_t(x,\boldsymbol{u}[\rho_t]),$$ for all $(x,\omega)$ and all $t\in[0,T]$. Since these convergences hold for all $(x,\omega)$, they hold for a.e. $(x,\omega)$ with respect to every measure $\mu_t$ for a.e. $t\in[0,T]$. Thus, the dominated convergence theorem, with the dominating function of the form $C(1+|\boldsymbol{u}[\rho_t]|^2)\in L^1(0,T; L^1_{\mu_t}(\mathbb{R}^{2d}))$ ({\it cf.} Lemma \ref{enstrophy}), allow us to pass to the limit on $\mathcal{L}_1$ and obtain
\begin{align}\label{l1}
{\mathcal L}_1\stackrel{\varepsilon\to 0 }{\longrightarrow} \int_0^T\int_{\RR^{2d}} \big(\partial_t\psi_t(x,\boldsymbol{u}[\rho_t]) + \boldsymbol{u}[\rho_t]\cdot\nabla_x\psi_t(x,\boldsymbol{u}[\rho_t])\Big)\, d\mu_t\, dt,
\end{align}
and, analogously for $\mathcal{R}$, we have
\begin{align}\label{r}
{\mathcal R}\stackrel{\varepsilon\to 0 }{\longrightarrow} - \int_{\RR^{2d}} \psi_0(x,\boldsymbol{u}[\rho_0]) \,d \mu_0.
\end{align}

\medskip

$\diamond$ {\sc Step 5}. Passing with $\varepsilon\to 0$; the case of ${\mathcal L}_2$.\\
It remains to prove the convergence of ${\mathcal L}_2$. We define the functions in the integrand of the symmetrized expression \eqref{E-L2-symmetrized} for $\mathcal{L}_2$, and in the expected limiting expression as $\varepsilon\rightarrow 0$, that is, the nonlinear term in \eqref{weakeqf}. Specifically, we define $h_\varepsilon$ and $h$ as follows
\begin{align*}
h_\varepsilon&:=\left(\nabla_v\psi_t(x,\boldsymbol{u}_\varepsilon[\rho_t]) - \nabla_v\psi_t(x,\boldsymbol{u}_\varepsilon[\rho_t])'\right)^\top D^2W_\varepsilon(x-x')\left(\boldsymbol{u}[\rho_t]'-\boldsymbol{u}[\rho_t]\right),\\
h&:=\left(\nabla_v\psi_t(x,\boldsymbol{u}[\rho_t]) - \nabla_v\psi_t(x,\boldsymbol{u}[\rho_t])'\right)^\top D^2 W(x-x')\left(\boldsymbol{u}[\rho_t]'-\boldsymbol{u}[\rho_t]\right)\,\mathds{1}_{\mathbb{R}^{4d}\setminus \Delta}.
\end{align*}
Note that in $h$ we are subtracting the diagonal set $\Delta$ defined in \eqref{diagonalset} in agreement with \eqref{weakeqf}. First, note that
$h_\varepsilon\rightarrow h$
whenever $x\neq x'$. When $x=x'$ there are two possibilities: either $\omega=\omega'$ or $\omega\neq \omega'$. In the former case, the point $((x,\omega),(x',\omega'))$ belongs to $\Delta$, where $h_\varepsilon$ simply reduces to $0$, and therefore it also converges to $0$ (thus to $h$) as $\varepsilon\rightarrow 0$. In the latter case, $((x,\omega),(x',\omega'))$ must belong to the collision set ${\rm Col}$, defined in \eqref{colisionset}, which is a negligible set with respect to the measure $\lambda$ by Lemma \ref{enstrophy}. Altogether, we obtain the following convergence as $\varepsilon\rightarrow 0$
$$h_\varepsilon\to h\quad \lambda\mbox{-a.e.},$$ 

We aim to apply here the dominated convergence theorem with respect to the measure $\lambda$, and then the rest of this step will focus on finding an appropriate domination by a $L^1$ function with respect to $\lambda$. Recall that $\nabla_v\psi_t$ is Lipschitz-continuous, which implies
\begin{align*}
|\nabla_v\psi_t(x,\boldsymbol{u}_\varepsilon[\rho_t]) - \nabla_v\psi_t(x,\boldsymbol{u}_\varepsilon[\rho_t])'|&\lesssim |x-x'| + |\boldsymbol{u}_\varepsilon[\rho_t]-\boldsymbol{u}_\varepsilon[\rho_t]'|,\\
|\nabla_v\psi_t(x,\boldsymbol{u}[\rho_t]) - \nabla_v\psi_t(x,\boldsymbol{u}[\rho_t])'|&\lesssim |x-x'| + |\boldsymbol{u}[\rho_t]-\boldsymbol{u}[\rho_t]'|,
\end{align*}
and by Young's inequality with exponent 2
\begin{align*}
    |h_\varepsilon|&\lesssim (|x-x'|+|\boldsymbol{u}_\varepsilon[\rho_t]-\boldsymbol{u}_\varepsilon[\rho_t]'|)|x-x'|^{-\alpha} |\boldsymbol{u}[\rho_t]-\boldsymbol{u}[\rho_t]'|\\
    &= \Big(|x-x'|^{1-\frac{\alpha}{2}}\Big)\Big(|\boldsymbol{u}[\rho_t]-\boldsymbol{u}[\rho_t]'||x-x'|^{-\frac{\alpha}{2}}\Big) + |\boldsymbol{u}_\varepsilon[\rho_t]-\boldsymbol{u}_\varepsilon[\rho_t]'||\boldsymbol{u}[\rho_t]-\boldsymbol{u}[\rho_t]'||x-x'|^{-\alpha}\\
    &\lesssim |x-x'|^{2-\alpha} + g + g_\varepsilon.
\end{align*}
By  Proposition \ref{W2cont}, and Lemmas \ref{enstrophy} and \ref{gepg}, the right-hand side above is dominated by an $\varepsilon$-independent $\lambda$-integrable function. Therefore, by dominated convergence theorem we have
\begin{align}\label{l2}
{\mathcal L}_2\stackrel{\varepsilon\to 0 }{\longrightarrow} \frac{1}{2}\int_0^T\int_{\RR^{4d}\setminus\Delta} \left(\nabla_v\psi_t(x,\boldsymbol{u}[\rho_t]) - \nabla_v\psi_t(x,\boldsymbol{u}[\rho_t])'\right)^\top D^2W (x-x')\left(\boldsymbol{u}[\rho_t]'-\boldsymbol{u}[\rho_t]\right) \,d\lambda.
\end{align}
Combining \eqref{l1}, \eqref{r} and \eqref{l2} with \eqref{diagl2} and returning to the original notation we arrive at the equation
\begin{align}\label{eqmu}
\int_0^T\int_{\RR^{2d}}\Big (\partial_t\psi_t(x,\omega-\nabla W*\rho_t(x)) + \boldsymbol{u}[\rho_t](x,\omega)\nabla_x\psi_t(x,\omega - \nabla W*\rho_t(x) )\Big)\,d\mu_t\,dt \nonumber\\
+ \frac{1}{2}\int_0^T\int_{\RR^{4d}\setminus\Delta} \left(\nabla_v\psi_t(x,\omega-\nabla W*\rho_t(x)) - \nabla_v\psi_t(x',\omega'-\nabla W*\rho_t(x'))\right)^\top\\
\times D^2 W(x-x')\left(\boldsymbol{u}[\rho_t](x',\omega')-\boldsymbol{u}[\rho_t](x,\omega)\right)\,d(\mu_t \otimes \mu_t)\,dt\nonumber\\
= - \int_{\RR^{2d}} \psi_0(x, \omega - \nabla W*\rho_0(x))\,d\mu_0\nonumber.
\end{align}

\medskip

$\diamond$ {\sc Step 6}. Conversion from $\boldsymbol{\mu}$ to $\boldsymbol{f}$.\\
It remains to observe that, by the very definition of $\boldsymbol{f}$ as a push-forward of $\boldsymbol{\mu}$ along the map $\cT^{1\to 2}$, the above equation \eqref{eqmu} is equivalent to
\begin{align*}
&\int_0^T\int_{\RR^{2d}}\Big(\partial_t\psi_t(x,v) + v\cdot\nabla_x\psi_t(x,v)\Big)\,df_t\,dt
\\
&\qquad + \frac{1}{2}\int_0^T\int_{\RR^{4d}\setminus\Delta} \Big(\nabla_v\psi_t(x,v) - \nabla_v\psi_t(x',v')\Big)^\top D^2 W(x-x')(v'-v) \,d(f_t\otimes f_t)\,dt\nonumber\\
&\qquad = - \int_{\RR^{2d}} \psi_0(x, v) \,df_0,
\end{align*}
which is precisely the weak formulation \eqref{weakeqf} for $\boldsymbol{f}$ according to Definition \ref{DCS} tested by an arbitrary admissible test function $\psi\in C^1_b([0,T]\times\RR^{2d})$ with compact support in $[0,T)$ and Lipschitz-continuous $\nabla_v\psi$.
\end{proof}

\subsection{From second order alignment dynamics to the first order Kuramoto system}\label{sec:CStoK}

The following result focuses on proving the converse implication to the one obtained in Proposition \ref{PCSexi}, then making up the full Theorem \ref{main1}. The proof follows by repeating analogous arguments from Proposition \ref{PCSexi}, and thus we provide only a sketch.

\begin{pro}\label{2ndto1st}
Suppose that $\alpha\in(0,1)$ and $T\geq 0$ and let $\boldsymbol{f}$ be a weak solution to \eqref{sysk} in the sense of Definition \ref{DCS}. Assume further that
\begin{align}\label{nudiagf}
    \sup_{t\in [0,T]}\int_{\RR^{2d}}\left\vert(v+\nabla W*\rho_t(x))-(v'+\nabla W*\rho_t(x'))\right\vert^\frac{2-3\alpha}{1-\alpha}d(f_t\otimes f_t)< \infty.
\end{align}
Then, $\mu_t=(\cT^{2\to 1}[\rho_t])_{\#} f_t$, with change of variables $\mathcal{T}^{2\to 1}[\rho_t]$ given in Definition \ref{D-change-variables}, is a weak solution of \eqref{kurak} in the sense of Definition \ref{D1st}.
\end{pro}

\begin{proof}[Sketch of the proof]
Our approach will be to invert all the argumentation introduced in the proof of Proposition \ref{PCSexi}, now aiming to show that $\mu_t=(\cT^{2\to 1}[\rho_t])_{\#} f_t$ satisfies \eqref{weakeqF} for any admissible test function $\eta\in C^1_b([0,T]\times \RR^{2d})$ compactly supported in $[0,T)$. 

Let us remark first that, without loss of generality, we can restrict to the class of special test functions $\eta\in C^1_b([0,T]\times \RR^{2d})$ satisfying the extra condition (which we will need below) that $\nabla_\omega\eta_t$ is Lipschitz-continuous uniformly in $t\in [0,T]$ . Specifically, once the weak formulation \eqref{weakeqF} is proved for any such $\eta$ as above, then it also holds for any $\eta\in C^1_c([0,T]\times \RR^{2d})$ because $\nabla_\omega\eta$ is bounded and uniformly-continuous, and therefore it can be approximated uniformly by bounded-Lipschitz functions \cite[Theorem 6.8]{H-01}. Hence, \eqref{weakeqF} must also hold for any $\eta\in C^1_b([0,T]\times \RR^{2d})$ by performing a simple cut-off argument as in the proof of Proposition \ref{W2contf}(iii) thanks to the uniform propagation of moments obtained in Proposition \ref{W2contf}(i). 

Let us fix then $\varepsilon>0$ and any $\eta\in C^1_b([0,T]\times \RR^{2d})$ compactly supported in $[0,T)$ such that $\nabla_\omega \eta_t$ is Lipschitz-continuous uniformly in $t\in [0,T]$. We pullback along $\cT^{1\to 2}$ setting
\begin{align*}
    \psi_t(x,v) &:= \eta_t(x,\boldsymbol{\omega}_\varepsilon[\rho_t](x,v)),\\
    \boldsymbol{\omega}_\varepsilon[\rho_t](x,v) &:=v+\nabla W_\varepsilon * \rho_t(x),
\end{align*}
where again $W_\varepsilon$ is the mollification defined in \eqref{epep}. By a similar reasoning to Lemma \ref{partialtu} based on Remark \ref{test}, we have the identity
\begin{align*}
    \partial_t(\boldsymbol{\omega}_\varepsilon[\rho_t]) = -[D^2W_\varepsilon*(vf_t)]:=-\int_{\RR^d}D^2W_\varepsilon(x -x') v'\, df_t,
\end{align*}
where again we use the notation $[M*V]$ introduced in \eqref{E-convolution-matrix-vector}. Using the chain rule implies
\begin{align*}
    \partial_t\psi_t(x,v) &= \partial_t\eta_t(x,\boldsymbol{\omega}_\varepsilon[\rho_t](x,v)) - \nabla_\omega\eta_t(x,\boldsymbol{\omega}_\varepsilon[\rho_t](x,v))\cdot [D^2W_\varepsilon*(vf_t)](x),\\
    \nabla_x\psi_t(x,v) &= \nabla_x\eta_t(x,\boldsymbol{\omega}_\varepsilon[\rho_t](x,v)) + \nabla_\omega\eta_t(x,\boldsymbol{\omega}_\varepsilon[\rho_t](x,v))^\top (D^2W_\varepsilon* f_t(x)),\\
    \nabla_v\psi_t(x,v) &= \nabla_\omega\eta_t(x,\boldsymbol{\omega}_\varepsilon[\rho_t](x,v)).
\end{align*}
Arguing as in the proof of Lemma \ref{partialtu}, all of the above terms belong to $C_b([0,T]\times\RR^{2d})$ (note that Proposition \ref{W2contf} plays a role here). It is also easy to infer that $\nabla_v\psi$ is Lipschitz-continuous, based on the mollification in $\boldsymbol{\omega}_\varepsilon$. Therefore $\psi\in C^1_b([0,T]\times\RR^{2d})$ is compactly supported in $[0,T)$ with Lipschitz-continuous $\nabla_v\psi$. Thus $\psi$ is an admissible test function in the weak formulation for $\boldsymbol{f}$ according to Definition \ref{DCS}, see also Proposition \ref{W2contf}(iii). Consequently, we have
\begin{multline*}
    \int_0^T\int_{\RR^{2d}} \left(\partial_t\psi + v\cdot\nabla_x\psi\right)\, df_t\, dt \\
    + \frac{1}{2}\int_0^T\int_{\RR^{4d}\setminus\Delta}(\nabla_v\psi-\nabla_v\psi')^\top D^2 W(x-x') (v'-v)\, d(f_t\otimes f_t')\, dt=-\int_{\RR^{2d}}\psi_0\, df_0,
\end{multline*}
which after expressing $\psi$ in terms of $\eta$ reads
\begin{multline}\label{1to2-1}
    \int_0^T\int_{\RR^{2d}}\Big( \partial_t\eta_t(x,\boldsymbol{\omega}_\varepsilon[\rho_t](x,v)) + v\cdot\nabla_x\eta_t(x,\boldsymbol{\omega}_\varepsilon[\rho_t](x,v))\Big)\, df_t\, dt\\
    + \int_0^T\int_{\RR^{4d}\setminus\Delta}\Bigg(\big(\nabla_\omega\eta_t(x,\boldsymbol{\omega}_\varepsilon[\rho_t](x,v))-\nabla_\omega\eta_t(x',\boldsymbol{\omega}_\varepsilon[\rho_t](x',v')\big)^\top\\
    ~\qquad \qquad \qquad \quad\   \times(D^2 W - D^2W_\varepsilon)(x-x') (v'-v)\Bigg)\, d(f_t\otimes f_t')\, dt=-\int_{\RR^{2d}}\eta_0(x,\boldsymbol{\omega}_\varepsilon[\rho_0](x,v))\, df_0.
\end{multline}
Next we pass with $\varepsilon$ to $0$. The key point is to converge with the second integral on the left-hand side of \eqref{1to2-1} to $0$. We follow the main ideas of {\sc Step 3} and {\sc Step 5} of Proposition \ref{PCSexi}:

\medskip

$\diamond$ {\sc Step 1}. We prove $\tilde\lambda$-a.e. convergence to $0$ of the integrand, where $\tilde\lambda$ is now defined as 
$$\tilde\lambda:=f_t(x,v)\otimes f_t(x',v')\otimes \mathcal{L}^1_{\mathlarger{\llcorner} [0,T]}(t).$$
Specifically, we infer pointwise convergence in the complement of the collision set ${\rm Col}$ defined in \eqref{colisionset}, which is $\tilde\lambda$-negligible under the energy dissipation assumption \eqref{weakeqfa}.

\medskip

$\diamond$ {\sc Step 2}. We dominate the integrand by an $L^1_{\tilde{\lambda}}([0,T]\times \mathbb{R}^{4d})$ function and use the dominated convergence theorem. To this end we employ a variant of Lemma \ref{gepg} with
\begin{align*}
    A = \Big\{(x,x',v,v')\in \mathbb{R}^{4d}:\,|v-v'|\leq\frac{3H}{1-\alpha}|x-x'|^{1-\alpha}\Big\},
\end{align*}
and replacing $\boldsymbol{u}_\varepsilon[\rho_t]$ with $\boldsymbol{\omega}_\varepsilon[\rho_t]$. Following the argument in {\sc Step 5} of the proof of Proposition \ref{PCSexi}, we upper-bound the intergrand in the middle term of \eqref{1to2-1} by the function
\begin{align*}
    |x-x'|^{2-\alpha} + |v-v'|^2\phi(|x-x'|) + \left\vert\boldsymbol{\omega}_\varepsilon[\rho_t](x,v)-\boldsymbol{\omega}_\varepsilon[ \rho_t](x',v')\right\vert|v-v'|\phi(|x-x'|),
\end{align*}
by virtue of the Lipschitz-continuity which we are assuming on $\nabla_\omega\eta$. A similar argumentation as in Lemma \ref{gepg} allows bounding such $\varepsilon$-dependent functions by a $\varepsilon$-independent function that belongs to $L^1_{\tilde{\lambda}}([0,T]\times \mathbb{R}^{4d})$ by assumptions \eqref{weakeqfa}, \eqref{weakeqfb} and \eqref{nudiagf}.

\medskip

 Ultimately, note that from the above we obtain the identity
\begin{align*}
        \int_0^T\int_{\RR^{2d}} \Big(\partial_t\eta_t(x,v+\nabla W*\rho_t(x)) + v\cdot\nabla_x\eta_t(x,v+\nabla W*\rho_t(x))\Big)\,df_t\,dt\\
    =-\int_{\RR^{2d}}\eta_0(x,v+\nabla W*\rho_0(x))\,df_0.
\end{align*}
Recalling that $\mu_t=(\cT^{2\to 1}[\rho_t])_{\#} f_t$ by Definition \ref{D-change-variables}, we then arrive at \eqref{weakeqF}, which is satisfied for all $\eta\in C^1_b([0,T]\times \mathbb{R}^{2d})$ compactly supported in $[0,T)$.
\end{proof}

\subsection{Proofs of Theorems \ref{main2} and \ref{main3}}\label{sec:Cfinalp}
Theorems \ref{main2} and \ref{main3} follow by applying Theorem \ref{main1} to the results of \cite{PP-22-1-arxiv} obtained for the first order system \eqref{kurak}. The main hassle amounts to a suitable interpretation of the results and translation into language more suited for the second order system \eqref{sysk}.
\begin{proof}[Proof of Theorem \ref{main2}]
Given any initial datum $f_0\in{\mathcal P}_2(\RR^{2d})$, we can push it forward along the map $\cT^{2\to 1}$, thus obtaining $\mu_0\in {\mathcal P}_2(\RR^{2d})$.  Then, thanks to \eqref{nudiag0t0}, condition \eqref{nudiag0} is fulfilled. Therefore, by \cite[Theorem 5.3]{PP-22-1-arxiv} there exists a unique solution $\boldsymbol{\mu}$ to \eqref{kurak} in the sense of Definition \ref{D1st}, which by Theorem \ref{main1} can be transformed back into a solution $\boldsymbol{f}$ to \eqref{sysk} in the sense of Definition \ref{DCS}. Thus existence is proved. Uniqueness under the a priori assumption \eqref{nudiagf0} follows by transforming solutions to \eqref{sysk} into solutions of \eqref{kurak}, using Theorem \ref{main1}, and then recalling uniqueness of solutions to \eqref{kurak} established in \cite{PP-22-1-arxiv}.
\end{proof}

\begin{proof}[Proof of Theorem \ref{main3}]
Under the assumptions of Theorem \ref{main2}, we have that the solutions $\boldsymbol{f}$ and $\widetilde{\boldsymbol{f}}$ can be transformed into solutions $\boldsymbol{\mu}$ and $\widetilde{\boldsymbol{\mu}}$ to \eqref{kurak} in the sense of Definition \ref{D1st}. These new solutions inherit the properties of $\boldsymbol{f}$ and $\widetilde{\boldsymbol{f}}$. Namely, the assumption \eqref{main30} we have
\[
    \int_{\RR^{2d}}x d\mu_0 = \int_{\RR^{2d}}x d\widetilde{\mu}_0,
\]
and from \eqref{diamsupp} we infer
\begin{align*}
    {\mathcal D}_x^\mu &:= \diam ({\rm supp}_x \mu_0) = {\mathcal D}_x,\quad
    \widetilde{{\mathcal D}}_x^\mu := \diam ({\rm supp}_x \widetilde{\mu}_0)= \widetilde{{\mathcal D}}_x,\\
    {\mathcal D}_\omega^\mu &:= \diam ({\rm supp}_\omega \mu_0) = {\mathcal D}_\omega,\quad
    \widetilde{{\mathcal D}}_\omega^\mu := \diam ({\rm supp}_\omega \widetilde{\mu}_0) = \widetilde{{\mathcal D}}_\omega,
\end{align*}
for $\mathcal{D}_x$, $\widetilde{\mathcal{D}}_x$, $\mathcal{D}_\omega$ and $\widetilde{\mathcal{D}}_\omega$ defined by \eqref{diamsupp} in terms of $f_0$ and $\widetilde{f_0}$. Then, $\boldsymbol{\mu}$ and $\widetilde{\boldsymbol{\mu}}$ satisfy the assumptions of Corollary 5.14 and Theorem 5.13 from \cite{PP-22-1-arxiv}, which imply the following.
\begin{itemize}
    \item There exists a unique compactly supported equilibrium $\mu_\infty\in \mathcal{P}_{2}(\mathbb{R}^{2d})$ of \eqref{kurak} with the same $\omega$-marginal $\nu$ as $\boldsymbol{\mu}$, such that
$$\int_{\mathbb{R}^{2d}}x\,d\mu_0(x,\omega) =\int_{\mathbb{R}^{2d}}x\,d\mu_\infty(x,\omega),$$
and we also have
\begin{equation}\label{finalchanges1}
W_{2}(\mu_t,\mu_\infty)\leq Ce^{-{2}\phi(D_1)t},
\end{equation}
for all $t\geq 0$, where we denote
$$D_1:=\max\left\{\mathcal{D}_x^\mu,\left(\mathcal{D}_\omega^\mu\right)^{\frac{1}{1-\alpha}}\right\}.$$
    \item Moreover, we have 
    $$
    AW_2(\mu_t,\widetilde{\mu}_t)\leq \left(1+\frac{1}{2\phi(D_2)}\right)\,AW_2(\mu_0,\widetilde{\mu}_0^2),
    $$
    for all $t\geq 0$, where we denote
    $$D_2:=\max\left\{\mathcal{D}_x^\mu,\widetilde{\mathcal{D}}_x^\mu,\big(\mathcal{D}_\omega^\mu\big)^{\frac{1}{1-\alpha}},\big(\widetilde{\mathcal{D}}_\omega^\mu\big)^{\frac{1}{1-\alpha}}\right\},$$
and $AW_2$ is the {\it adapted Wasserstein-2 metric} defined as
\begin{equation*}
AW_2(\mu_1,\mu_2):=\left(\inf_{\hat{\nu}\in \Gamma(\nu_1,\nu_2)}\int_{\mathbb{R}^{2d}} \left(W_2^2(\mu_1^\omega,\mu_2^{\omega'})+\vert \omega-\omega'\vert^2\right)\,d\hat{\nu}(\omega,\omega')\right)^{1/2},
\end{equation*}
also called {\it nested} or {\it causal Wasserstein distance}, see \cite{L-18,PP-12,PP-14}. Here, $\nu_i:=\pi_{\omega\#}\mu_i$ and  $\{\mu_i^\omega\}_{\omega\in \mathbb{R}^d}$ is the associated disintegration introduced in \eqref{disint} for $i=1,2$.
\end{itemize}

\medskip

In order to finish the proof it remains to express the above information for $\boldsymbol{\mu}$ and $\widetilde{\boldsymbol{\mu}}$ in the language of $\boldsymbol{f}$ and $\widetilde{\boldsymbol{f}}$. This is done in the following steps:

\medskip

$\diamond$ {\sc Step 1}. Convergence of $\boldsymbol{f}$ to equilibrium.\\
The equilibrium $\mu_\infty$ satisfies assumptions of Theorem \ref{main1} (particularly it inherits assumption \eqref{nudiag0} from $\boldsymbol{\mu}$, which in turn inherits it from $\boldsymbol{f}$). Thus,
taking $f_\infty:= (\cT^{1\to 2}[\rho_\infty])_\#\mu_\infty$ with $\rho_\infty = \pi_{x\#}\mu_\infty$ we immediately see that $f_\infty$ is an equilibrium of \eqref{sysk} and satisfies \eqref{main30}. It remains to properly upper-bound $W_2(f_t,f_\infty)$ in terms of $W_2(\mu_t,\mu_\infty)$. Let $\gamma_t$ be an optimal transference plan between $\mu_t$ and $\mu_\infty$. Then $\gamma^f_t:=(\cT^{1\to 2}[\rho_t],\cT^{1\to 2}[\rho_\infty])_\#\gamma_t$ is an admissible plan connecting $f_t$ with $f_\infty$. Thus, using \eqref{finalchanges1} in the last line below,
\begin{align*}
    W_2^2(f_t,f_\infty)&\leq \int_{\RR^{4d}}|x-x'|^2 + |v-v'|^2 \,d\gamma^f_t\\
    &=  \int_{\RR^{4d}}|x-x'|^2 + |\omega-\omega' + \nabla W*\rho_\infty(x) - \nabla W*\rho_t(x')|^2 \,d\gamma_t\\
    &\lesssim \int_{\RR^{4d}}|x-x'|^2 + |\omega-\omega'|^2 d\gamma_t + \int_{\RR^{4d}} |x-x'|^{2-2\alpha} \,d\gamma_t\\
    &\lesssim  W_2^2(\mu_t,\mu_\infty) + W_2^{2-2\alpha}(\mu_t,\mu_\infty)\\
    &\lesssim e^{-4\phi(D_1)t} + e^{-4(1-\alpha)\phi(D_1)t}\lesssim e^{-4(1-\alpha)\phi(D_1)t},
\end{align*}
for every $t\geq 0$, and therefore \eqref{main31} is proven.

\medskip

$\diamond$ {\sc Step 2}. Uniform stability.\\
We simply define for all $\sigma_1$ and $\sigma_2$ in ${\mathcal P}_2(\RR^{2d})$ the distance
$$
{\rm dist}(\sigma_1,\sigma_2):= AW_2((\cT^{2\to 1}[\rho_1])_\#\sigma_1, (\cT^{2\to 1}[\rho_2])_\#\sigma_2),
$$
where we define $\rho_1 := \pi_{x\#}\sigma_1$ and $\rho_2 := \pi_{x\#}\sigma_2$. Verifying that ${\rm dist}(\cdot,\cdot)$ defines indeed a distance follows from the fact that $AW_2(\cdot,\cdot)$ is a distance and that the transformation $\sigma\mapsto (\cT^{2\to 1}[\rho])_\#\sigma$ for $\rho=\pi_{x\#}\sigma$ is bijective, as explained in the paragraph below Definition \ref{D-change-variables}.
\end{proof}


\appendix

\section{Properties of potential $W$}\label{Appendix-convexity-W}

In this part, we summarize the main convexity properties of the potential function $W$ along with some useful identities and bounds which are used throughout the paper. We start with the following technical lemma.

\begin{lem}\label{szwarc}
The following properties hold true
\begin{enumerate}[label=(\roman*)]
    \item (Convexity: first-order condition)
    $$(\nabla W(x)-\nabla W(y))\cdot (x-y)\geq \Lambda_1(x,y,\alpha)\vert x-y\vert^2.$$
    \item (Convexity: second order condition)
    $$v^\top\ D^2 W(x)\ v\geq \Lambda_2(x,\alpha)\vert v\vert^2.$$
    \item (Outwards pointing gradient)
    $$x\cdot\nabla W(x) = \frac{1}{1-\alpha}|x|^{2-\alpha}\geq 0.$$
    \item (Bound of the Hessian)
    $$|D^2 W(x)|\leq \frac{1+\alpha}{1-\alpha}\phi(|x|).$$
\end{enumerate}
for every $x,x'\in \mathbb{R}^d$ and $v\in \mathbb{R}^d$. Here, the functions $\Lambda_1$ and $\Lambda_2$ take the form
$$\Lambda_1(x,y,\alpha):=\int_0^1 \phi(\vert (1-\tau)x+\tau y\vert)\,d\tau,\quad \Lambda_2(x,\alpha):=\phi(\vert x\vert).$$
Since $\Lambda_1,\Lambda_2\geq 0$, the potential $W$ is in particular convex.
\end{lem}

\begin{proof}
Recall that by direct computation we have
\begin{align}
\nabla W(x) &= \frac{1}{1-\alpha}\phi(|x|)\,x,\label{E-W-gradient}\\
D^2 W(x) &= \frac{1}{1-\alpha}\phi(|x|)\left(I_d-\alpha \frac{x}{|x|}\otimes \frac{x}{|x|}\right),\label{E-W-Hessian}
\end{align}
Then, $(ii)$ follows from \eqref{E-W-Hessian} by the Cauchy--Schwartz inequality and $(i)$ follows from $(ii)$ by a first-order Taylor expansion. Moreover, $(iii)$ is trivial in view of \eqref{E-W-gradient}, and $(iv)$ follows from \eqref{E-W-Hessian} by the triangle inequality for the matrix norm.
\end{proof}

Consequently, using the $0$-convexity of $W$ and the Cauchy-Schwarz inequality we obtain one-sided Lipschitz continuity of the velocity field associated with equation \eqref{kurak}. More specifically, we obtain the following set of properties.

\begin{cor}\label{C-sided-Lipschitz}
Consider any $\rho\in \mathcal{P}(\mathbb{R}^d)$ with $\int_{\mathbb{R}^d}|x|^{1-\alpha}\,d\rho(x)<\infty$, and let $\boldsymbol{v}[\rho]=(\boldsymbol{u}[\rho],0)$ be the velocity field of the Kuramoto-type equation \eqref{kurak}, that is,
$$\boldsymbol{u}[\rho](x,\omega)=\omega-\nabla W*\rho(x).$$
Then, $\boldsymbol{v}[\rho]$ verifies the following conditions:
\begin{enumerate}[label=(\roman*)]
\item (Continuity) $\boldsymbol{v}[\rho]\in C(\mathbb{R}^{2d},\mathbb{R}^{2d})$.
\item (One-sided Lipschitz) We have
$$(\boldsymbol{v}[\rho](x,\omega)-\boldsymbol{v}[\rho](x',\omega'))\cdot((x,\omega)-(x',\omega'))\leq \frac{1}{2}\vert (x,\omega)-(x',\omega')\vert^2,$$
for any $x,x',\omega,\omega'\in \mathbb{R}^d$.
\item (Sublinear growth) We have
$$\vert \boldsymbol{v}[\rho](x,\omega)\vert\leq \frac{1}{1-\alpha}M_{1-\alpha}(\rho)+\frac{1}{1-\alpha}\vert x\vert^{1-\alpha}+\vert \omega\vert,$$
for any $x,\omega\in \mathbb{R}^d$, with $M_{1-\alpha}(\rho):=\int_{\mathbb{R}^{2d}}|x|^{1-\alpha}\,d\mu(x,\omega).$
\end{enumerate}
\end{cor}

\begin{proof}
$\diamond$ {\sc Step 1}. Proof of $(i)$.\\
Let us take $(x_n,\omega_n)\rightarrow (x,\omega)$ in $\mathbb{R}^{2d}$ and note that
$$\boldsymbol{u}[\rho](x_n,\omega_n)-\boldsymbol{u}[\rho](x,\omega)=(\omega_n-\omega)-\int_{\mathbb{R}^d}\nabla W(x-x')\,d(\rho_n-\rho_t)(x'),$$
where $\rho_n:=\tau_{x-x_n\#}\rho$, and $\tau_{x-x_n}$ denotes the translation operator with shift $x-x_n$. Note that by definition we have that $\rho_n\rightarrow \rho$ narrowly, but $\nabla W$ is unbounded though. To prove convergence to zero of the second term  above we need to ensure (see \cite[Lemma 5.1.7]{AGS-08}) that the function 
$$x'\in \mathbb{R}^d\longmapsto \vert \nabla W(x-x')\vert \lesssim 1+|x'|^{1-\alpha},$$
is uniformly integrable with respect to $\{\rho_n\}_{n\in \mathbb{N}}$, that is, 
$$\lim_{R\rightarrow \infty}\sup_{n\in \mathbb{N}}\int_{|x'|\geq R}|x'|^{1-\alpha}\,d\rho_n(x')=0.$$
However, this is clear because we have
\begin{align*}
\int_{|x'|\geq R}|x'|^{1-\alpha}\,d\rho_n(x')&=\int_{|x'|\geq \frac{R}{2}}|x'+(x-x_n)|^{1-\alpha}\,d\rho(x')\\
&\lesssim |x_n-x|^{1-\alpha}\int_{|x'|\geq \frac{R}{2}}d\rho(x')+\int_{|x'|\geq \frac{R}{2}}|x'|^{1-\alpha}d\rho(x'),
\end{align*}
and the two factors in the right hand side vanish as $R\rightarrow \infty$ uniformly in $n\in \mathbb{N}$ because $\{x_n\}_{n\in \mathbb{N}}$ is bounded and $\int_{\mathbb{R}^d}1+|x'|^{1-\alpha}\,d\rho(x')<\infty$ by hypothesis.

\medskip

$\diamond$ {\sc Step 2}. Proof of $(ii)$.\\
By definition we have
\begin{align*}
(\boldsymbol{v}[\rho]&(x,\omega)-\boldsymbol{v}[\rho](x',\omega'))\cdot ((x,\omega)-(x',\omega'))\\
&=(\omega-\omega')\cdot(x-x')-\int_{\mathbb{R}^d} (\nabla W(x-z)-\nabla W(x'-z))\cdot (x-x')\,d\rho(z)\\
&\leq \frac{1}{2}(|x-x'|^2+|\omega-\omega'|^2),
\end{align*}
for every $(x,\omega),(x',\omega)\in \mathbb{R}^{2d}$, where in the last step we have used the Cauchy-Schwarz and Young inequalities in the first term, and the convexity property of $W$ in Lemma \ref{szwarc}(i).

\medskip

$\diamond$ {\sc Step 3}. Proof of $(iii)$.\\
This part is straightforward by definition of $\boldsymbol{u}[\rho]$ since we have
\begin{align*}
|\boldsymbol{v}[\rho](x,\omega)|&=\left\vert \omega-\int_{\mathbb{R}^d}\nabla W(x-x')\,d\rho(x')\right\vert\\
&\leq |\omega|+\frac{1}{1-\alpha}\int_{\mathbb{R}^d}|x-x'|^{1-\alpha}\,d\rho(x')\\
&\leq \frac{1}{1-\alpha}\int_{\mathbb{R}^d}|x'|^{1-\alpha}\,d\rho(x')+\frac{1}{1-\alpha}|x|^{1-\alpha}+|\omega|,
\end{align*}
for every $(x,\omega)\in \mathbb{R}^{2d}$.
\end{proof}

We remark that the above regularity is key for the well-posedness of a unique global flow of $\boldsymbol{v}[\rho]$, and solving \eqref{kurak} via the method of characteristics, see \cite{P-19-arxiv}. More specifically, we have.

\begin{cor}[Lagrangian solutions]\label{C-Lagrangian-solution}
Consider any $T>0$ and $\mu_0\in \mathcal{P}_2(\mathbb{R}^{2d})$, and let $\boldsymbol{\mu}$ be any weak measure-valued solution to \eqref{kurak} with initial datum $\mu_0$ in the sense of Definition \ref{D1st}. Let $Z(t;x,\omega)=(X(t;x,\omega),\omega)$ be the characteristic flow, which solves
\begin{align}\label{E-characteristic-system}
\begin{aligned}
&\frac{d}{dt}X(t;x,\omega)=\boldsymbol{u}[\rho_t](Z(t;x,\omega)), && \mbox{a.e.}\ t\in [0,T],\\
&X(0;x,\omega)=x, && (x,\omega)\in \mathbb{R}^{2d}.
\end{aligned}
\end{align}
Then, \eqref{E-characteristic-system} admits a global Carath\'eodory solution for any $(x,\omega)\in \mathbb{R}^{2d}$, which is unique forward-in-time. Moreover, the weak measure-valued solution $\boldsymbol{\mu}$ is Lagrangian, {\it i.e.}, it verifies
\begin{equation}\label{E-lagrangian-solution}
\mu_t=Z(t;\cdot,\cdot)_{\#}\mu_0,\quad t\in [0,T].
\end{equation}
\end{cor}

\begin{proof}
We divide the proof into two parts.

\medskip

$\diamond$ {\sc Step 1}. Well posedness of \eqref{E-characteristic-system}.\\
For simplicity of notation, we set the velocity fields 
$$\boldsymbol{u}_t(x,\omega):=\boldsymbol{u}[\rho_t](x,\omega)\quad \mbox{and}\quad \boldsymbol{v}_t(x,\omega):=(\boldsymbol{u}_t(x,\omega),0).$$
We shall employ Carath\'eodory's existence theorem to prove the existence of Carath\'eodory solutions to the characteristic system \eqref{E-characteristic-system}, that is, absolutely continuous trajectories solving the above system for almost every $t\in [0,T]$. To do so, we need to verify the following three conditions in Carath\'eodory's existence theorem:
\begin{enumerate}[label=(\roman*)]
    \item (Time-measurability) The map $t\in [0,T]\mapsto \boldsymbol{v}_t(x,\omega)$ is measurable for each $(x,\omega)\in \mathbb{R}^{2d}$.
    \item (Space-continuity) The map $(x,\omega)\in \mathbb{R}^{2d}\mapsto \boldsymbol{v}_t(x,\omega)$ is continuous for {\it a.e.} $t\in [0,T]$.
    \item (Domination by $L^1$ function) For every $(x_0,\omega_0)\in \mathbb{R}^{2d}$ there exists $R>0$ such that 
    $$t\in [0,T]\mapsto \Vert\boldsymbol{v}_t(x,\omega)\Vert_{L^\infty(B_R(x_0)\times B_R(\omega_0))}\in L^1(0,T).$$
\end{enumerate}
The time-measurability in $(i)$ is clear, and the spacial-continuity in $(ii)$ follows from Corollary \ref{C-sided-Lipschitz}(i). Then, we only focus on the domination by a $L^1$ function of time. Consider any $(x_0,\omega_0)\in \mathbb{R}^{2d}$ and $R>0$ and let us use the sublinear growth in Corollary \ref{C-sided-Lipschitz}(iii) to obtain
$$\vert \boldsymbol{v}_t(x,\omega)\vert\lesssim |x_0|^{1-\alpha}+|\omega_0|+R^{1-\alpha}+R+\int_{\mathbb{R}^{2d}}|x|^{1-\alpha}\,d\mu_t(x,\omega),$$
for any $(x,\omega)\in B_R(x_0)\times B_R(\omega_0)$ and each $t\in [0,T]$. Note that the the right hand side belongs to $L^1(0,T)$ because of condition \eqref{E-weakeqF} in Definition \ref{D1st}. Hence, there is at least one local-in-time Carath\'eodory solution to \eqref{E-characteristic-system} for each $(x,\omega)\in \mathbb{R}^{2d}$. The solution is in fact globally defined in $[0,T]$ by Gr\"{o}nwall's lemma due to the sublinear growth of $\boldsymbol{v}_t$ and, in addition, it is unique forward-in-time in view of the one-sided Lipschitz condition in Corollary \ref{C-sided-Lipschitz}(ii).

\medskip

$\diamond$ {\sc Step 2}. Lagrangian solution.\\
To prove that the weak measure-valued solution $\boldsymbol{\mu}$ to \eqref{kurak} must propagate along the forward-unique characteristic flow in the above step, we follow the approach by {\sc Ambrosio} and {\sc Crippa} in \cite{AC-14}. Specifically, by definition we note that $\boldsymbol{\mu}$ is a solution to the continuity equation \eqref{kurak} in distributional sense \eqref{weakeqF} with given velocity field $\boldsymbol{v}_t$. In addition, such a velocity field verifies appropriate integrability conditions, namely,
$$\int_0^T\int_{\mathbb{R}^{2d}}\frac{|\boldsymbol{v}_t(x,\omega)|}{1+|x|+|\omega|}\,d\mu_t(x,\omega)\,dt<\infty,$$
by the sublinear growth in Corollary \ref{C-sided-Lipschitz}(iii) and the assumption \eqref{E-weakeqF} in Definition \ref{D1st}. Hence, the {\it superposition principle} in \cite[Theorem 12]{AC-14} implies that $\boldsymbol{\mu}$ can be represented as a {\it superposition solution}. Since the characteristic flow \eqref{E-characteristic-system} is well posed by the previous considerations, then we conclude that $\boldsymbol{\mu}$ must be Lagrangian, {\it i.e.}, \eqref{E-lagrangian-solution} holds true.
\end{proof}

\section{Two useful theorems}

In the final appendix we present two useful theorems: the classical disintegration theorem \cite[Theorem 5.3.1]{AGS-08} and the local-in-time mean-field limit, which follows from similar arguments to those developed by the second author in \cite{P-19-arxiv} for the singular Kuramoto model.

\begin{theo}[Disintegration]\label{dis}
Set $d_1,d_2\in \mathbb{N}$, define the projection onto the second component $\pi_2:\mathbb{R}^{d_1}\times \mathbb{R}^{d_2}\longrightarrow \mathbb{R}^{d_2}$ and consider $\mu\in \mathcal{P}(\mathbb{R}^{d_1}\times \mathbb{R}^{d_2})$ and $\nu:=(\pi_2)_\# \mu\in \mathcal{P}(\mathbb{R}^{d_2})$. Then, there exists a $\nu$-a.e. uniquely defined family of probability measures $\{\mu^{x_2}\}_{x_2\in \mathbb{R}^{d_2}}\subseteq \mathcal{P}(\mathbb{R}^{d_1})$ so that the following properties hold true:
\begin{enumerate}[label=(\roman*)]
\item For every Borel subset $B\subseteq \mathbb{R}^{d_1}$, the following map is Borel-measurable
$$
x_2\in \mathbb{R}^{d_2}  \longmapsto  \mu^{x_2}(B).
$$
\item For every Borel-measurable map $\varphi:\mathbb{R}^{d_1}\times \mathbb{R}^{d_2}\longmapsto [0,+\infty)$, we obtain
\begin{align}\label{disb}
\int_{\mathbb{R}^{d_1}\times \mathbb{R}^{d_2}} \varphi(x_1,x_2)\,d\mu(x_1,x_2)=\int_{\mathbb{R}^{d_2}}\left(\int_{\mathbb{R}^{d_1}}\varphi(x_1,x_2)\,d\mu^{x_2}(x_1)\right)\,d\nu(x_2).
\end{align}
\end{enumerate}
\end{theo}

\begin{pro}[Local-in-time mean field limit]\label{P-mean-field-local-in-time} 
Consider any initial datum $\mu_0\in \mathcal{P}_2(\mathbb{R}^{2d})$ and let $\boldsymbol{\mu}$ be the (unique) weak measure-valued solution of \eqref{kurak} issued at $\mu_0$. Assume that $\boldsymbol{\mu}^N$ is a sequence of empirical measures over the phase-space configurations $(x_1^N(t),\omega_1^N)$, $\ldots\,$, $(x_N^N(t),\omega_N^N)$ solving the particle system \eqref{kurap} such that
$$\lim_{N\rightarrow\infty} W_2(\mu_0^N,\mu_0)=0.$$
Then, we obtain the local-in-time mean field limit
$$\lim_{N\rightarrow\infty}\sup_{t\in [0,T]}W_2(\mu_t^N,\mu_t)=0,$$
for every $T>0$.
\end{pro}

\bibliographystyle{amsplain} 
\bibliography{PP-22}

\end{document}